\documentclass[12pt, a4paper]{amsart}
\usepackage{amscd, amssymb}
\usepackage{tikz-cd}

\usepackage{enumitem}
\usepackage{tensor}
\usepackage{wrapfig}
\usepackage{tikz}

\allowdisplaybreaks

\def\tight{\operatorname{tight}}
\def\chara{\operatorname{char}}

\def\supp{\operatorname{supp}}
\def\ker{\operatorname{ker}}

\def\Iso{\operatorname{Iso}}
\def\Tr{\operatorname{Tr}}

\def\C{\mathbb{C}}
\def\F{\mathbb{F}}
\def\K{\mathbb{K}}

\def\N{\mathbb{N}}
\def\Z{\mathbb{Z}}

\def\AA{\mathcal{A}}

\def\FF{\mathcal{F}}
\def\G{\mathcal{G}}

\def\JJ{\mathcal{J}}

\def\NN{\mathcal{N}}

\def\Gf{\mathfrak{G}}
\def\Hf{\mathfrak{H}}

\newcommand{\bfz}{\mathbf{0}}
\newcommand{\bfo}{\mathbf{1}}

\newtheorem{thm}{Theorem}[section]

\newtheorem{lemma}[thm]{Lemma}

\theoremstyle{definition}

\theoremstyle{remark}
\newtheorem{remark}[thm]{Remark}

\newtheorem{example}[thm]{Example}

\numberwithin{equation}{section}

\tikzstyle{vertex}=[circle]
\tikzstyle{goto}=[->,shorten >=1pt,>=stealth,semithick]

\begin{document}

\date{June 28, 2024}
\title{Simplicity of  $*$-algebras of non-Hausdorff $\Z_2$-multispinal groupoids}
\author{C. Farsi}
\email{carla.farsi@colorado.edu}
\author{N.~S. Larsen}
\email{nadiasl@math.uio.no}
\author{J. Packer}
\email{packer@colorado.edu}
\author{N. Thiem}
\email{nathaniel.thiem@colorado.edu}
\maketitle


\begin{abstract}
	We study simplicity of $C^*$-algebras arising from self-similar groups of $\Z_2$-multispinal type, a generalization of the Grigorchuk case whose simplicity was first proved by L. Clark, R. Exel, E. Pardo, C. Starling, and A. Sims in 2019, and we prove results generalizing theirs. Our first main result is a sufficient condition for simplicity of the Steinberg algebra  satisfying conditions modeled on the behavior of the groupoid associated to the first Grigorchuk group. This closely resembles conditions found by B. Steinberg and N. Szak\'acs. As a key ingredient we identify an infinite family of $2-(2q-1,q-1,q/2-1)$-designs, where $q$ is a positive even integer.  We then deduce the simplicity of the associated $C^*$-algebra, which is our second main result.  Results of similar type were considered by B. Steinberg and N. Szak\'acs in 2021, and later by K. Yoshida, but their methods did not follow the original methods of the five authors.
	\end{abstract}
\section{Introduction}

Groupoids and their $C^*$-algebras have attracted a lot of attention in recent years, and  the particular case of \'etale groupoids has especially been a topic of focus.
\'Etale groupoids have been associated to foliations and other geometric settings.
While the Hausdorff groupoid case has been the object of many of the first papers, there has also been a lot of recent interest in non-Hausdorff groupoids with Hausdorff unit space, see e.g.
\cite{Exel-Pitts}.

Here we consider in detail the groupoids associated to a class of multispinal groups.
Multispinal groups were introduced by Steinberg and Szak\'acs, who studied simplicity of their Nekrashevych algebra over fields of
varying characteristic and showed that (non)-simplicity varies with the value of the characteristic,  \cite{Stei-Sza-Contr}. Later, Yoshida provided alternative characterisations of the simplicity of the associated groupoid $C^*$-algebras using methods based on ideas in \cite{Stei-Sza-Contr} and a characterization of KMS states, \cite{Y}.

 Multispinal groups provide valuable models of self-similar actions, and we here concentrate our attention on the class of \v{S}uni\'{c} groups constructed from primitive polynomials of arbitrary degree $n\geq 2$ over the field $\F_2$ with two elements; we call them $\Z_2$-multispinal groups, see \cite{Su} for the general construction. Of particular note is the polynomial $f_2(x)=x^2+x+1$, leading to the Grigorchuk group.

We proceed with a quick review of self-similar actions and their associated groupoids. Self-similar groups play an important role in geometric group theory. One of the main examples of self-similar groups is the Grigorchuk group constructed from an action based on the rooted tree on the alphabet $X=\{0,1\}$. The Grigorchuk group  is the first known example of a finitely generated group of
intermediate growth \cite{Gri1}; see Nekrashevych \cite{Nekra-Crelle} \cite{Nekra-Sao-Paulo},\cite{Nekra-growth},\cite{Nekra} among others. Roughly speaking, self-similar actions of groups are defined in the following way.   Let $X$ be a finite set with more than one element, called the alphabet, and let $X^*$ denote the set of finite words over $X$. One may view $X^*$ as a rooted tree, with the empty word as root, and an edge between $\eta$ and $\eta x$ for every word $\eta$ in $X^*$ and letter $x$.
If we have a faithful length-preserving action of $G$ on $X^*$, written
$(g,\eta)\;\to\; g\cdot \eta,$ such that for all $g \in G, x \in X,$ there exists a unique element of
$G,$ denoted $g|_{x}$ such that for all $\eta \in X^*,$
$$g\cdot (x\eta)\;=\;(g \cdot x)(g|_{x}
\cdot \eta),$$
the pair $(G,X)$ is called a {\it self-similar action}.   The map $G \times X\;\to\; G,\;(g, x) \;\to g|_{x},$
is called the {\it restriction}. This property  gives the action an additional structure that parallels the self-similarity of $X^*$ given by the shift map $x \eta \to \eta$ which projects first-level subtrees to the whole tree.

Nekrashevych defined a  $C^*$-algebra associated to self-similar groups, see \cite{Nekra-Crelle}.  One can also construct a groupoid naturally associated to self-similar actions  with Hausdorff unit space but which is not necessarily Hausdorff. This provides another starting point for constructions of $C^*$-algebras associated to self-similar actions. Steinberg defined an algebraic *-algebra associated to an ample (and \'etale) groupoid $\mathcal{G}$ with coefficients in a field $\K$, henceforth referred to as its \emph{Steinberg algebra},  \cite{Steinberg-A}. In the case of groupoids from self-similar actions, the Steinberg algebra is an algebraic version of the Nekrashevych  $C^*$-algebra, see e.g. \cite{CEPSS}, \cite{Stei-Sza-Contr} and the references therein. We refer to Section \ref{sub:prelim} for details on both of these algebras. Moreover, for ample non-necessarily Hausdorff groupoids,  in the paper \cite{CZ} they give an alternative  way to construct from the Steinberg algebra the groupoid $C^*$-algebra.

The starting point motivating this research was the study of simplicity of the Steinberg algebra and the $C^*$-algebra of the groupoid associated with the self similar action of the first Grigorchuk group on $\{0,1\}$ from \cite{CEPSS}. The main abstract result for a not necessarily Hausdorff $\G$  with Hausdorff unit space $\G^{(0)}$ is \cite[Theorem~ 4.10]{CEPSS}. In particular, for a second-countable, locally compact, \'etale groupoid such that
$\G^{(0)}$ is Hausdorff, if $C^*(\G)$ is simple, then $C^*(\G) = C^*_r(\G)$, $\G$ is effective and for every nonzero $a\in C^*_r(\G)$, $\supp(j(a))$ has non-empty interior. In \cite{CEPSS}, an element $a\in C^*_r(\G)$ is \textit{singular} if $\supp(j(a))$ has empty interior.

Trying to understand the work of \cite{CEPSS} quickly lead us to wonder whether other self-similar groups apart from the Grigorchuk group displayed the key features of singular functions  expressed as combinations of characteristic functions in a family depending on the nucleus of the self-similar group, cf. \cite[Section 5]{CEPSS}. Of course, similar questions were asked by other authors. Steinberg and Szak\'acs provided a sweeping generalization of the simplicity result in the case of Nekrashevych algebras of self-similar actions, and that in particular covered the case of multispinal groups, see \cite[Section 7]{Stei-Sza-Contr} and especially their Theorems 7.6 and 7.10. Their methods of proof build on an algebraic approach in two main parts, one identifying a simplicity graph, more precisely  a Schreier graph, associated to the self-similar action, with the second  step being a representation theoretic characterisation of  conditions for (non)-simplicity.

In this work, we study simplicity of $C^*$-algebras of non-Hausdorff groupoids, here denoted $\mathcal{G}_n$, arising from self-similar groups of $\Z_2$-multispinal type, here denoted $\Gf_n$ where $n\geq 2$, a generalization of the Grigorchuk case, and prove results generalizing those of \cite{CEPSS}, see Sections~\ref{section:simplicity} and \ref{section:simplicity C*-algebra}.
The proofs of our main simplicity results, for the Steinberg algebra $\AA_\C(\mathcal{G}_n)$ and the groupoid $C^*$-algebra $C^*(\mathcal{G}_n)$  follow closely the thread laid out in \cite{CEPSS} in a series of lemmas. This approach relies more on topology and linear algebra than group theory.

The key contribution of the present work is that it provides a bridge, via considerations of  matrix ranks related to $2$-designs in \cite{GLL}, \cite{Wil} and  \cite{Assmus}, between the methods of proofs for \cite[Theorem 5.22(1)]{CEPSS} and for \cite[Theorem 7.10]{Stei-Sza-Contr} in  case of the prime $p=2$. Perhaps most worthy of attention is the distinction in simplicity depending on the characteristic of the underlying field. As we read from \cite{GLL} and \cite{Wil}, this distinction has appeared elsewhere in a different context. A natural question is whether similar considerations of matrix ranks  are true for other prime values, but here we do not have an answer, as we have not identified analogous $p$-designs for odd values of the prime $p$.

Another approach to simplicity of groupoid Steinberg and $C^*$-algebras associated to multispinal groups is due to Yoshida, see \cite{Y}. The necessary and sufficient condition for simplicity of the universal $C^*$-algebra associated to an amenable self-similar action of a multispinal group in \cite[Theorem 3.13]{Y} is given in terms of invertibility of a certain matrix associated to the group.  Since the first Grigorchuk group is such an important construction, with ramifications in different fields, we think it is worthwhile to have access to different perspectives on the (non)-simplicity of Steinberg and groupoid $C^*$-algebras of this and similar groups.

The paper is structured as follows: In Section \ref{sec:back-Gpds}  we first review  $C^*$-algebras from self-similar actions  and how they can be viewed as groupoid $C^*$-algebras, where the groupoids are associated to inverse semigroups arising from self-similar actions. In Section~\ref{section:applications} we start by describing the construction of the $\Z_2$-multispinal groups $\Gf_n$ of interest here. We identify an explicit family of full rank matrices associated to a
\[
2-\bigl(2^n-1,2^{n-1}-1, 2^{n-2}-1 \bigr)-\text{design},
\]
which we then show to be a special case of a  $2-(2q-1,q-1,q/2-1)$-design, where $q$ is a positive even integer. We prove simplicity of the Steinberg algebra of $\mathcal{G}_n$ over a field of characteristic zero in Section~\ref{section:simplicity}, and in Section~\ref{section:simplicity C*-algebra} we present our results on the simplicity of the full $C^*$-algebra associated to $\mathcal{G}_n$. The main bulk of the two last sections consists of versions of the lemmas obtained in \cite[Section 5]{CEPSS} by working with the specific generators of the first Grigorchuk group (which we recall is the multispinal group $\Gf_2$)  that are the elements of the nucleus fixing the letters of the alphabet $X$. So our effort goes into lifting these results to an arbitrary  $\Z_2$-multispinal group  $\Gf_n$ with $n\geq 3$. We hope that some of the methods developed here may extend to study simplicity, and even nosimplicity, in the case of other self-similar groups.
Moreover, according to very recent research just posted on the arxiv (see \cite{BKMc}), our methods also imply simplicity  of the associated groupoid $L^p$-operator algebras (similarly to the case of the Grigorchuk group, see \cite[Example]{BKMc}).

\emph{Acknowledgments: }{This research was partly supported by:  Simon Foundation Collaboration Grant \#523991 (C.F), Simon Foundation Collaboration Grant \#1561094 (J.P.) and  the Trond Mohn Research Foundation through the project "Pure Mathematics in Norway" (N.L.). We thank Zahra Afsar and Lisa O. Clark for useful discussions at the very early stages of this project, when the first three named authors participated in the workshop "Women in Operator Algebras" at BIRS (Banff International Research Centre), 4-9 November 2018, meeting \# 18w5168. We acknowledge fruitful discussions with Nora Szak\'acs at the late stages of this research.}

\section{Background on ample groupoids and their *-algebras}
\label{sec:back-Gpds}

In this section we recall some basic facts about
self-similar actions and  their associated groupoids. We also recall their connection with groupoids associated to inverse semigroups.
Then   we introduce the self-similar actions groupoid  and their Steinberg and C*-algebras. The ideal of singular functions plays an important role in simplicity settings.

At the end of this section we list a criterion for a set to have non-empty interior in the non-Hausdorff setting.

\subsection{Preliminaries on groupoids}
\label{sub:prelim}
Suppose that $\G$ is a topological groupoid with unit space $\G^{(0)}$ and source and range maps $s,r$. An \emph{open bisection $\Theta$} is an open subset $B\subset \G$ such that the source map (equivalently, the range map) restricts to a homeomorphism of $B$ onto an open subset of $\G$. In this paper, all  groupoids are assumed to be second countable and have a Hausdorff unit space $\G^{(0)}$.

The groupoid $\G$ is \emph{{\'e}tale} if its topology admits a basis of open bisections and is \textit{minimal} if for any $u \in \G^{(0)}$ the orbit $[u] := s(r^{-1}(u))$ of $u$  is dense in $\G^{(0)}$. For each $u\in \G^{(0)}$, consider the isotropy group and form  the isotropy bundle
\[\G_u^u:=\{\gamma\in \G:r(\gamma)=s(\gamma)\text{ and } \Iso(\G):=\bigcup_{u\in \G^{(0)}}\G_u^u.\]
The groupoid  $\G$ is \textit{effective } if the  interior of $\Iso(\G)$
is equal to $\G^{(0)}$. We have that $\G^{(0)}$ embeds as an open subset of $\Iso(\G)$.
The groupoid $\G$ is \textit{topologically principal} if the collection of units with trivial isotropy group is
dense in $\G^{(0)}$. If the groupoid is \'etale and second countable, then being effective implies being
 topologically principal. It is known that the converse only holds for Hausdorff groupoids.

\subsection{Tight groupoid associated to an inverse semigroup}
\label{sec:tight-groupoid}

We say a semigroup $S$  is \textit{regular}  if for all $s \in S$ there exists an element $t \in S$ that satisfies $tst = t$ and $sts = s$. The element $t$ is called an \textit{inverse} of $s$.  A regular semigroup is called an \textit{inverse semigroup} if each
element $s$ has a unique inverse  $s^*$. It is well-known that  a regular semigroup is an
inverse semigroup if and only if elements of $E(S):=\{e \in S: e^2 = e\}$ commute.  Note that in this case
$E(S)$ is closed under multiplication. Also $E(S)$ is a semi-lattice with order given by
\[e\leq f \iff ef=e.\]
Given an inverse semigroup $S$,  let $s\in S$ and $ e\in E(S)$ such that
$e\leq s^*s$. Then   $e$ is  \textit{weakly fixed} by $s$  if $sfs^*f \neq 0$ for every nonzero idempotent $f \leq e$.

Let $S$ be an inverse semigroup. A subset $\xi \subset E(S)$ is \textit{downwards directed} if whenever $e, f \in \xi$ then we have  $ef \in \xi$. A subset $\xi \subset E(S)$
is \textit{upwards closed}  if $e\in \xi, f\in E(S)$
and $e\leq f$ implies that $f \in \xi$.
A \textit{filter} in $E(S)$ is a proper subset $\xi \subset E(S)$ that is both  downwards directed and upwards closed. A filter is called an \textit{ ultrafilter} if it is not properly
contained in another filter.

We write $\hat{E}_0(S)$ for the  set of filters in $E(S)$. Then the subspace $\hat{E}_{0}(S)\subseteq\{0, 1\}^{E(S)}$  has the product topology.
For finite sets $X, Y \subset E(S)$, let
\[U(X, Y ) := \{\xi\in E^0(S) : x \in\xi \text{ for all  }x \in X, y \notin \xi \text{ for all  }y \in Y\}.\]
These sets are clopen and generate the subspace topology on $\hat{E}_0(S)$ as $X$ and $Y$
range over all finite subsets of $E(S)$. We also write $ \hat{E}_\infty(S)$ for the  subspace of ultrafilters. We
will denote by $\hat{E}_{\tight}(S)$ the closure of $\hat{E}_\infty(S)$ in $\hat{E}_0(S)$ and call this the space of \textit{tight
	filters}.

Given an inverse semigroup $S$, there is a natural action of $S$ on $\hat{E}_{\tight}(S)$:  Let $ D_e:= \{\xi\in \hat{E}_{\tight}(S): e\in\xi\}$, and
define $\theta_s: D_{s^*s} \to D_{ss^*}$ by
\begin{equation}
\label{eq:theta-bisections}\theta_s(\xi) = \overline{\{ses^* : e\in \xi\}}.
\end{equation}
We write $\G_{\tight}:=\G(S,\hat{E}_{\tight},\theta)$ for the groupoid of germs of this action and call it the \textit{tight groupoid} of $S$.

Let  $(\{D_e\}_e\in E(S), \{\alpha_ s\}_s\in S) $ be an action of  an
inverse semigroup $S$ on a locally compact Hausdorff space $X$, and let $s \in S$. The set
\[F_s := \{x\in  X : \alpha_s(x) = x\}
\]
is called the set of \textit{fixed points for $ s$}. Also let
\[T F_s :=\{x \in X : \text{ there exists }e\in E(S)  \text{ such that } 0\neq e\leq s \text { and }x \in D_e\}=\bigcup_{e\leq s}D_e.\]
and call this the set of \textit{trivially fixed points for $s$}.

\subsection{Groupoids arising from self-similar actions} Let $X$ be a finite set with more than one
element, let $G$ be a group, and let $X^*$ denote the set of all words in elements of
$X,$ including an empty word $\emptyset.$ We denote by $X^{\omega}$ the Cantor set of one-sided infinite
words in $X,$ given the product topology.
Recall that the cylinder sets $C(\eta) = \{\eta x : x\in X^{\omega}\}$ form a clopen basis for the topology on
$X^{\omega}$ as $\eta$ ranges over $X^*$.

If we have a faithful length-preserving action of $G$ on $X^*$, written
$(g,\eta)\;\to\; g\cdot \eta,$ such that for all $g \in G, x \in X,$ there exists a unique element of
$G,$ denoted $g|_{x}$ such that for all $\eta \in X^*,$
$$g\cdot (x\eta)\;=\;(g \cdot x)(g|_{x}
\cdot \eta),$$
the pair $(G,X)$ is called a {\it self-similar action}.   The map $G \times X\;\to\; G,\;(g, x) \;\to g|_{x},$
is called the {\it restriction}, and extends to $G \times X^*$ via the formula
$$g|_{x_1\cdot x_ n}\;=\; g|_{x_1}\;|_{x_2}\cdots \;|_{x_n}.$$
For all $\eta,\;\mu\;\in X^*,$ and all $g\in G,$ it is the case that
$$g\cdot (\eta\mu)\;=\;(g\cdot \eta) ( g|_{\eta}\cdot \mu).$$
The action of $G$ on $X^*$ extends to an action of $G$ on $X^{\omega}$
defined by
$$g\cdot (x_1x_2x_3\ldots )\;=\;(g\cdot x_1) (g|_{x_1}\cdot x_2)(g|_{x_1x_2}\cdot x_3)\;\ldots \;.$$

As defined in \cite{Nekra}, the Nekrashevych $C^*$-algebra ${\mathcal O}_{(G,X)}$ associated to a self similar action $(G,X)$ is the universal $C^*$-algebra generated by isometries $\{s_x\}_{x\in X}$ and a unitary representation $\{u_g\}_{g\in G}$ satisfying expected Cuntz-like relations and the covariance equation $u_g s_x\;=\; s_{g\cdot x} u_{g|_{x}},\;\forall g\in G$ and $x\in X.$  It so happens that ${\mathcal O}_{(G,X)}$ is isomorphic to the $C^*$-algebra of the tight groupoid associated to the inverse semigroup $S_{G,X},$ where
$$S_{G,X}\;=\{(\eta, g, \mu): \eta,\mu \in X^*, g\in G\}\cup \{0\},$$
and
$$(\eta, g, \mu) (\gamma, h, \nu)\;=\;\left\{ \begin{array}{cc}  (\eta (g\cdot \epsilon), g|_{\epsilon} h,\nu ) & \;\text{if}\; \gamma=\mu\epsilon, \\   (\eta, g (h^{-1}|_{\epsilon})^{-1},\nu (h^{-1}\cdot \epsilon ) & \;\text{if}\;  \mu=\gamma\epsilon, \\
0& \; \text{otherwise}, \\
\end{array} \right.$$
with
$$(\eta, g, \mu)^*\;=\;(\mu, g^{-1}, \eta).$$
The set of idempotents of $S_{G,X}$ is given by $E(S_{G,X})$ computed to be:
$$E(S_{G,X})\;=\;\{(\mu, 1_G, \mu): \mu\in X^*\}\cup \{0\}.$$
It is known that the tight spectrum of $E(S_{G,X})$ ia homeomorphic to $X^{\omega},$ and the standard action of $S_{G,X}$ on this tight spectrum has been determined to be as follows:  $\theta(\eta, g, \mu):C(\mu)\;\to\;C(\eta);\;\text{for}\; \theta(\eta, g, \mu) (\mu \omega)\;=\;\eta (g\cdot \omega).$   See Section \ref{sec:tight-groupoid} for details.

As in the previous subsection, we thus obtain the ample, minimal groupoid ${\mathcal G}_{G,X}$ defined to be ${\mathcal G}_{\text{tight}}(S_{G,X}).$  The action of $G$ on $X^*$ is faithful, so that ${\mathcal G}_{G,X}$ is effective, and it is known that the Nekrashevych $C^*$-algebra ${\mathcal O}_{(G,X)}$ is isomorphic to the groupoid $C^*$-algebra $C^*({\mathcal G}_{G,X}),$ \cite{EP}.  We shall henceforth assume this fact.

\subsection{The Steinberg algebra of an ample groupoid and the ideal of singular functions}
Let $\G$ be a topological groupoid. Following \cite{CEPSS}, the groupoid $\G$ is \emph{ample} if it is \'etale and admits a basis for the topology consisting of compact open bisections. If $\G$ is an ample groupoid and $\K$ is a field, then the Steinberg algebra of $\G$ with coefficients in $\K$ is the set
\begin{equation}\label{eq:def Steinberg algebra}
A_\K(\G)=\operatorname{span}\{1_B\mid B \text{ is a compact open bisection}\}
\end{equation}
with the operation of convolution of functions on $\G$.
An element $f$ in $A_\K(\G)$ has form $f=\sum_{B\in \FF}a_B1_B$, where $\FF$ is a finite collection of compact open bisections and $a_B\in \K$. As indicated in \cite[Section 3.1]{CEPSS}, see also \cite[Remark 2.5]{CFST}, there is an alternative description of $f$ resulting from  disjointifying  $\FF$. Moreover, sets of the form described below have appeared elsewhere, see in particular the proof of \cite[Theorem 4.2]{Nekra-growth}. Explicitly, for each $\emptyset \neq \JJ\subseteq \FF$ let
\begin{equation}\label{eq:M disjointified}
M_\JJ=(\bigcap_{B\in \JJ}B)\setminus (\bigcup_{D\notin \JJ}D).
\end{equation}
Then $\bigcup_{F\in \FF}F$ is the disjoint union  $\bigcup_{\JJ}M_\JJ$, where we can express the latter by listing $\JJ$ according to its size, $1\leq |\JJ|\leq |\FF|$. Further, for each $\emptyset \neq \JJ\subseteq \FF$ set
\[
c_\JJ=\sum_{B\in \JJ} a_B=f\vert_{M_{\JJ}}.
\]
In all, we have
\begin{equation}\label{eq:function disjointification}
f=\sum_{\emptyset \neq \JJ\subseteq \FF}c_\JJ1_{M_\JJ}.
\end{equation}
It follows that a function in $A_\K(\G)$ can be nonzero exactly on the union of finitely many sets given as in \eqref{eq:M disjointified}.

 A function $f$ in $A_\K(\mathcal{G})$ is \emph{singular}, see \cite[Section 3]{CEPSS},  provided that
\[
f=\sum_{i=1}^n c_i1_{S_i}
\]
where each $S_i$ is a relatively closed subset of some open bisection $O_i$, and thus of the form $S_i=O_i\cap C_i$ with $C_i$ a closed set, such that  $\overset{\circ}{S_i}=\emptyset$.

Given $f\in A_\K(\mathcal{G})$, its support is defined as $\operatorname{supp}(f)=\{x\in \G\mid f(x)\neq0\}$. Note that the support is not necessarily closed.
It is known that $f\in A_\K(\mathcal{G})$ is singular if and only if $\operatorname{supp}(f)$ has empty interior, \cite[Proposition 3.6]{CEPSS}. Further, the set $S_\K(\mathcal{G})$ of singular functions is an ideal of $A_\K(\mathcal{G})$, see \cite{CEPSS} for the case that $\mathcal{G}$ is second-countable {\'e}tale and has Hausdorff unit space and \cite{Stei-Sza-Simpl} in general.

We shall assume that $\mathcal{G}$ is \emph{countably} ample, in the sense that it is an \'etale groupoid that has a countable base of compact open bisections.  In particular, for
each $z\in\mathcal{G}^{(0)}$ there is a neighborhood basis $(U_m(z))_{m\in \N}$ consisting of compact open bisections. In our examples we will  have that $U_m(z)\subseteq U_n(z)$ when $m\geq n$ for each $z$ in a finite subset of the unit space $\mathcal{G}^{(0)}$. (This does not follow automatically in the non-Hausdorff case.)

The authors of \cite{CEPSS} identify important characterizations of when the support of a function in $A_\K(\mathcal{G})$ is non-empty by means of the notion of \emph{regular open set}. We next detail some material that explores  the structure of non-Hausdorff groupoids. In particular the notion of regular open set is pivotal.
The section below is not needed for the sequel of the paper, but we thought it interesting as it generalizes sufficient conditions for the existence of empty interior in Lemma \ref{lem:slight-improv-Lemma-2.1}.

\subsection{Regular open sets} We recall from \cite[Section 9]{GH} that an open set $P$ in a topological space ${X}$ is called \emph{regular open} provided that $P$ is equal to the interior of its own closure. For a subset $A$ of ${X}$, we denote the complement of $A$ by $A'$, the interior of $A$ by $A^\circ$ and the closure of $A$ by $A^{-}$. There are well-known identities:
\[
A^{-}=((A')^\circ)'\text{ and }A^\circ=((A')^{-})',
\]
or, simply written  $A^{-}={{A'}^{\circ}}'$ and $A^\circ={{A'}^{-}}'$. As in \cite{GH}, we let $A^\perp={A^{-}}'$. Given a subset $P$ of $\mathcal{X}$, we have $P^{-}={P^\perp}'$, so
\[
{P^{-}}^\circ=({P^\perp}')^\circ=({P^\perp}^{-})'=P^{\perp\perp}.
\]
In other words, an open set is regular open precisely when $P=P^{\perp\perp}$. Note that the forward implication is always true, thus the non-trivial claim for a regular open set is that $P^{\perp\perp}\subset P$.

An intersection of regular open sets is again regular open. By contrast, a \emph{union of regular open sets need not be regular open}. The main reason this happens is that a typical open set that is not regular open arises as an open set with "cracks", cf. \cite[Section 89, page 61]{GH}. For example, $\mathbb{R}^2$ with finitely many points (or a line) removed will be an open set that is not regular open. When taking the union of regular open sets, a nonregular set may arise if the sets have cracks between them, cf. \cite[Section 10, page 67]{GH}.

In the case that $X$ is a topological groupoid $\G$ that is second countable, ample and effective, then every compact open bisection is regular open, \cite[Lemma 2.5]{CEPSS}. In general, it can happen that a union of compact open bisections is not regular open, as shown for the case of the Grigorchuk group in \cite[Section 5]{CEPSS}.

It was observed in \cite[Lemma 2.1]{CEPSS} that if $B$ and $D$ are regular open sets in $X$ such that $B\setminus D$ is nonempty, then $B\setminus D$ has nonempty interior. This result is the main ingredient in providing the characterisation in \cite[Lemma 3.1]{CEPSS} of when a function in $A_K(\mathcal{G})$  has nonempty interior. Since in the application to (ample effective) groupoids the set $B$ will be a finite intersection of compact open bisections, and thus regular open, while $D$ will be a finite union of compact open bisections, and thus not necessarily regular open, it is natural to relax the assumptions in \cite[Lemma 2.1]{CEPSS} by dropping the requirement that $D$ should be regular open.

To this end, note that we have the following identity  for all open sets $B$ and $D$ in a (not necessarily Hausdorff) topological space $X$:
\begin{equation} \label{eq:splitting the B over D and D perp perp}
B \setminus D =\Big(  B \setminus D^{\perp \perp} \Big) \bigsqcup \Big(  B \cap \big( D^{\perp \perp}\setminus D \big)\Big).
\end{equation}
Clearly, if $D$ is regular open, then \eqref{eq:splitting the B over D and D perp perp} is trivial. Thus the identity is interesting in the case that $D\subsetneq D^{\perp\perp}$.

	\begin{lemma}\label{lem:slight-improv-Lemma-2.1} Let $B$ be a regular open set, and $D$ an open set  such that $B \cap D$ is regular open. If  $B \backslash D$ is nonempty, then $B\setminus D$ has nonempty interior. In fact, if $B\setminus D^{\perp\perp}$ is nonempty, then $B\setminus D^{\perp\perp}$ and hence $B\setminus D$ have nonempty interior.	\end{lemma}
	\begin{proof} The proof of the first claim is verbatim as for \cite[Lemma 2.1]{CEPSS} once we observe that regular openness of $B$ and $B\cap D$ implies that $\overline{D \cap B} \subsetneq  \overline{B}$. This provides the open set $O$ equal to the complement of $\overline{D \cap B}$ which satisfies  $\emptyset \not= B \cap O \subset  B \backslash D $. The last claim follows by an application of \cite[Lemma 2.1]{CEPSS} to the regular open sets $B$ and $D^{\perp\perp}$.
	\end{proof}

We note that the requirement to have $B\cap D$ regular open when $B$ is so in Lemma~\ref{lem:slight-improv-Lemma-2.1} reduces to $B\cap D=(B\cap D)^{\perp\perp}=B^{\perp\perp}\cap D^{\perp\perp}=B\cap D^{\perp\perp}$, by an application of \cite[Lemma 4]{GH} about distributing $^{\perp\perp}$ over finite intersections of open sets.

\section{$\Z_2$-Multispinal groups and the structure of their subgroups of of order $2^{n-1}$}
	\label{section:applications}
	
In this section we will first recall the definition of the $\Z_2$-multispinal groups, which are the focus of our study. In particular they are related to the finite groups associated to the additive structure of finite fields of order $2^n$. After that, we  will detail the structure of  subgroups of order $2^{n-1}$ and associated inclusion matrices. An algebraic lemma at the end of this section shows these inclusion matrices to have full rank.

	\subsection{$\Z_2$-Multispinal groups from primitive polynomials over $\Z_2$}
	\label{def:multispinal}

	We now consider a class of examples of self-similar groups $\mathfrak{G}_n$ introduced by \cite{Su} and shown to be multispinal groups from primitive polynomials by Steinberg and  Szak\'acs, see especially \cite[Theorem 7.10]{Stei-Sza-Contr} for the complete description of the (non)-simplicity pattern of the associated Nekrashevych algebra. We will call them \emph{$\Z_2$-multispinal groups}.
	
	Two useful examples to keep in mind arise from the low-degree primitive polynomials $f_2(x)=x^2+x+1$ and $f_3(x)=x^3+x+1$ over $\Z_2$; the self-similar group associated to $f_2(x)$ is the first Grigorchuk group, see example \ref{ex:Grigo}. In all of these cases the groupoid $\mathcal{G}_n$ associated to the group $\mathfrak{G}_n$  lends itself  to a characterisation of the simplicity pattern of its Steinberg algebra $\mathcal{A}_K(\mathfrak{G}_n)$ and full groupoid $C^*$-algebra $C^*(\mathcal{G}_n)$ using insight from combinatorial design theory, as we will explain in Sections~\ref{section:simplicity} and \ref{section:simplicity C*-algebra}.
	
	In the following, let $\F_2$ be the field with two elements $0,1$ and $f_n(x)$ be a primitive polynomial of degree $n\geq 2$ over $\F_2$. Let $\alpha$  be a zero of $f_n(x)$ (note that we have not added the index $n$ to $\alpha$ to make the notation simpler), and let $\F_2[x]/\langle f_n(x)\rangle$ be the quotient field, identified with the field $\F_{2^n}$ in the standard way
	\[
	\F_{2^n}=\{\beta_0+\beta_1\alpha+\dots +\beta_{n-1}\alpha^{n-1}\mid b_0,\dots, b_{n-1}\in \F_2\}.
	\]
	We remark that as an additive group, $\F_{2^n}$ can be identified with the vector space $(\F_2)^n $ of dimension $n$ over $\F_2.$
	Since $f_n(x)$ is  primitive we have that $\alpha$ is a generator of the cyclic group $\F_{2^n}^*$ on $2^n-1$ elements.
	Let $\varphi_n$ be the automorphism of $\F_{2^n}$ determined by multiplication by $\alpha$. Under the identification
	
	\[
	(\beta_0,\beta_1,\dots,\beta_{n-1})\mapsto \beta_0+\beta_1\alpha+\dots +\beta_{n-1}\alpha^{n-1}
	\] of $\Z_2^n$ with $\F_2^n$ as additive groups, the map $\varphi_n$ determines an automorphism $\phi_n$ of $\Z_2^n$ whose effect on $(0,1,0,\cdots,0)$ is to produce $\alpha^2$; continued application of $\phi_n$ gives  the table of all the powers $\alpha^j$, $j=1,\dots, 2^n-1$.  The map $\phi_n$ is known as a "feedback shift register". For  $f_3(x)$ it is given by $(\beta_0,\beta_1,\beta_2)\mapsto (\beta_2, \beta_0+\beta_2, \beta_1)$. Feedback shift registers associated with primitive polynomials of arbitrary degree $n\geq 1$  are of interest in generating pseudo-random numbers, and are used among others in error correcting codes, see for example \cite[Section 5.6]{Ter}.

Consider the trace  map
	$\Tr:\F_{2^n}\to \Z_2$, defined for all $\beta \in \F_{2^n}$, by:
	 \[\Tr(\beta):=\beta+\beta^2+\dots + \beta^{2^{n-1}}.\]
	 (To simplify the notation, we will not add the index $n$ to $\Tr$ as it will be clear from the context what $n$ is.) It is well-known that $\Tr$ is a  homomorphism of abelian groups, see e.g. \cite[Chp. V, section 7]{Hun}.
	
In order to distinguish in an easier way between group elements and the space they act on, we let $X$ be an alphabet on two letters $\bfz,\bfo$. Following \cite[Section 7.1]{Stei-Sza-Contr}, form the multispinal group  $\mathfrak{G}_n$ associated to $G_n :=\F_{2^n}$ (under addition), $H=\Z_2$ acting on itself by the left regular representation $\lambda$, with $\lambda_1$ swapping the basis elements $\delta_0$ and $\delta_1$ of $\ell^2(\Z_2)$, and the map $\Phi:\{\bfz,\bfo\}\to \operatorname{Aut}({G}_n)\cup \operatorname{Hom}({G}_n, H)$ given by $\Phi(\bfz):=\Tr$ and $\Phi(\bfo):=\varphi_n$. In the sequel, these will be called \emph{$\Z_2$-multispinal groups}. So $\mathfrak{G}_n$ contains a copy of $\F_{2^n}$, identified as $\Z_2^n$, and of $\Z_2$ in such a way that the identities $0_{G_n}$ and $0_H$ are identified as the identity $e$ in $\mathfrak{G}_n$. Furthermore, we have that $\bigcap_{j=0}^{2^n-2}\ker(\Tr\circ\varphi_n^j)=\{0\}$.
	
	We will now explain how these groups can be viewed as
	self-similar.
	Denote by $a$  the image in $\mathfrak{G_n}$ of $1\in \Z_2$.  In the self-similar group $\mathfrak{G}_n$ the action of $a$ on $X$ replicates that of $1\in H$ on itself, thus
	\[a\cdot\bfz=\bfo, a\cdot\bfo=\bfz.
	\]

 For every $\beta\in \F_{2^n}$ let $\beta\cdot x=x$ for $x=\bfz,\bfo$ and define the  restrictions in $\mathfrak{G}_n$ by $g\vert_\bfz=g\vert_\bfo=e$ for $g\in \Z_2$ and
	\[
	\beta\vert_\bfo:=\varphi_n(\beta)=\alpha\beta\text{ and }\beta\vert_\bfz:=\Tr(\beta).
	\]
	Let
	 \[ \iota_n:\F_{2^n}\to \mathfrak{G}_n\]
	  be the embedding which identifies $x\in \F_{2^n}$ with its image in $\mathfrak{G}_n$. Let $\mathcal{N}_{0,n}$ be the set  consisting of $e$ and $\iota_n(\alpha^j)$, with $j=1,\dots,2^n-1$:
	  \[
	  \mathcal{N}_{0,n} := \{e, \iota_n(\alpha^j), \  j=1,\dots,2^n-1\}.
	  \]
	   By \cite{Stei-Sza-Contr}, the \emph{nucleus of $\mathfrak{G}_n$} is $\mathcal{N}:=\mathcal{N}_{0,n}\sqcup \{a\}.$ Extend the restriction rules by setting, for $g=\iota_n(x)\in \mathfrak{G}$,
	\begin{equation}\label{eq:restriction to one}
	g\vert_{\bfo}:=\iota_n(\varphi_n(x)) \text{ and }g\vert_{\bfz}:=\iota_n(\Tr(x)).
	\end{equation}
	It is immediate that $g\vert_{\bfo^j}=\iota_n(\varphi^j_n(x))$  for all $j\in \N$.
	
\begin{remark} It is natural to ask whether $\mathfrak{G}_n$ depends on the primitive polynomial $f_n(x)$. Let us look at the case $n=3$. There are two primitive polynomials over $\mathbb{F}_2$ of degree dividing $n$, let us denote them $f_{3,1}(x)=x^3+x+1$ and $f_{3,2}(x)=x^3+x^2+1$. Note that
\[
x^8-x=x(x-1)(x^3+x+1)(x^3+x^2+1).
\]
If we let $\alpha$ be a zero of $f_{3,1}(x)$, then $\alpha^2, \alpha^4$ are the other zeros of this polynomial and $\alpha^3, \alpha^5,\alpha^6$ are the zeros of $f_{3,2}(x)$. In particular, the subset $\mathcal{N}_{0,3}$ of the nucleus of  $\mathfrak{G}_3$ does not depend on the primitive polynomial, since it consists of the images of the zeros of $x^8-x$. More generally, $x^{2^n}-x$ admits a factorization over $\F_2$ as the product of all the distinct irreducible polynomials of degrees that divide $n$, see e.g. \cite{Ter}. This suggests that the simplicity of the Steinberg algebra of $\mathfrak{G}_n$ over a field $\mathbb{K}$ of characteristic zero does not depend on the primitive polynomial that we pick at the start.
\end{remark}

	\subsection{Structure of subgroups of order $2^{n-1}$}
	
	We will now  detail the structure of  subgroups of order $2^{n-1}$ of the $\Z_2$-multispinal groups $\mathfrak{G}_n$.
	
	For each $j=0,\dots,2^n-2$, define
\begin{equation}\label{eq:subgroups Hj}
  H_{j,n}:=\ker(\Tr\circ \varphi_n^j);
\end{equation}
note that these are $2^n-1$ distinct subgroups  of $\F_{2^n}$ of order $2^{n-1}$.  Also note that viewing $\F_{2^n}$ as the vector space $(\F_2)^n$ of dimension $n$ over $\F_2,$ then the $\{H_{j,n}\}_{j=0}^{2n-2}$ are precisely the $2^n-1$ distinct subspaces of dimension $(n-1)$ over $\F_2.$  Another way of parametrizing these subspaces is to take all non zero vectors $\{\vec{b}=(b_1,b_2,\dots,b_n)| b_i\in \F_2\}$ in $(\F_2)^n.$  Let $D_{\vec{b}}$ be the homomorphism from $(\F_2)^n$ onto $\F_2$ given by the dot product: $D_{\vec{b}}(\vec{x})= \vec{x}\cdot \vec{b}.$  Then the subspaces of dimension  $n-1,\;\{H_{j,n}\}_{j=0}^{2n-2},$ each of order $2^{n-1},$ are in one-to-one correspondence with the subgroups  $\{\text{ker}(D_{\vec{b}}):
	\;\vec{b}\in\;(\F_2)^n\backslash\;\{\vec{0}\}\},$ which are subgroups of index $2$ so also each have order $2^{n-1}.$
	
	\begin{lemma}
		\label{Key_intersection_ala_Nat}
		Consider all the subgroups $\{H_{j,n}\}_{j=0}^{2^n-2}$ of order $2^{n-1}$ of $\F_{2^n}$ described above. Let $q=2^{n-1}$ and $k=2^n-1$ so that $k=2q-1.$  Let $\{\alpha^0=1,\alpha^1,\alpha^2,\cdots, \alpha^{2^n-2}=\alpha^{2q-2}=\alpha^{k-1}\}$ denote all the non-zero elements in $\F_{2^n}^*.$  Fix any two distinct integers $\ell_1,\ell_2\in\{0,1,\dots,2q-2=k-1=2^n-2\}.$
		Then
		$$|\{j\in \{0,1,\dots, 2^n-2=k-1\mid \;\{\alpha^{\ell_1}, \alpha^{\ell_2}\}\subset H_{j,n}\}|= 2^{n-2}-1=\frac{q}{2}-1.$$
	\end{lemma}
	\begin{proof} To simplify the notation in  this proof we call $H_{j,n}$ simply $H_j$.
		We note that
\begin{align*}
c_j:=&|\{j\in \{0,1,\dots, 2^n-2=k-1\}\mid \;\{\alpha^{\ell_1}, \alpha^{\ell_2}\}\subset H_j\}|\\
&=|\{\vec{\beta}\in (\F_2)^n\backslash\{\vec{0}\}\mid \;\{\alpha^{\ell_1}, \alpha^{\ell_2}\}\subset \text{ker}(D_{\vec{\beta}})\}|
\end{align*}
		since we have set up a one-to-one correspondence between the additive subgroups $\{H_{j}\}_{j=0}^{2n-2}$ of $F_{2^n}$ and the $(n-1)$-dimensional subspaces  	$\text{ker}(D_{\vec{\beta}})_{\vec{\beta}\in (\F_2)^n\backslash\{\vec{0}\}}.$
		But now note that letting $\vec{e_1}=(1,0,\dots,0)\in (\F_2)^n,$ and $\vec{e_2}=(0,1,0,\dots,0)\in (\F_2)^n,$ we have
\begin{align*}
c_j&=|\{\vec{\beta}\in (\F_2)^n\backslash\{\vec{0}\}\mid \;\{\alpha^{\ell_1}, \alpha^{\ell_2}\}\subset \text{ker}(D_{\vec{\beta}})\}|\\
&=|\{\vec{\beta}\in (\F_2)^n\backslash\{\vec{0}\}\mid \;\{\vec{e_1}, \vec{e_2}\}\subset \text{ker}(D_{\vec{\beta}})\}|.
\end{align*}

		This is by symmetry of $(\F_2)^n:$ for $n\geq 2,$ we can find a linear transformation $T\in GL(n,\F_2)$ that takes $\alpha^{\ell_1}$ to $\vec{e_1}$ and  $\alpha^{\ell_2}$ to $\vec{e_2}.$  The effect of the linear transformation $T$ on the $(n-1)$-dimensional subspaces is to permute them amongst themselves.  Therefore,  we get the equality of cardinality amongst  $(n-1)$-dimensional subspaces.
		Finally,
		\begin{align*}
			c_j&=|\{\vec{\beta}\in (\F_2)^n \backslash\{\vec{0}\} \mid \;\{\vec{e_1}, \vec{e_2}\}\subset \text{ker}(D_{\vec{\beta}})\}|\\
&=|\vec{\beta}\in (\F_2)^n\backslash\{\vec{0}\}\mid\;\vec{\beta}\cdot \vec{e_1}=0\;\text{and}\; \vec{\beta}\cdot \vec{e_2}=0\}|\\
			&=\;|\{\vec{\beta}=(\beta_1,\beta_2,\dots, \beta_n) \in (\F_2)^n\backslash\{\vec{0}\}\mid\;\vec{\beta}\cdot \vec{e_1}=0\;\text{ and }\; \vec{\beta}\cdot \vec{e_2}=0\}|\\
		&=\;|\{(\beta_1,\beta_2,\dots, \beta_n) \in (\F_2)^n\backslash\{\vec{0}\}\mid\;\beta_1=0 \text{ and }\beta_2=0\}|=2^{n-2}-1\\
&=\frac{q}{2}-1.
		\end{align*}
		We thus obtain the claim of the Lemma.
			\end{proof}	

The following lemma will be key  in showing some of our simplicity claims.

	\begin{lemma}
		\label{primitivepolynomial_lemma_over_F2}
		Let $\F_2$ be the field with two elements $0,1$ and $f_n(x)$ be a primitive polynomial of degree $n\geq 2$ over $\F_2$. Let $\alpha$ be a zero of $f_n(x)$, and let $\F_2[x]/\langle f_n(x)\rangle$ be the quotient field, identified with the field $\F_{2^n}$ in the standard way
		\[
		\F_{2^n}=\{\beta_0+\beta_1\alpha+\dots +\beta_{n-1}\alpha^{n-1}\mid b_0,\dots, b_{n-1}\in \F_2\}.
		\]
		Let $q=2^{n-1}$ and $k=2^n-1$ so that $k=2q-1.$  Let $\{0,1=\alpha^0,\alpha^1,\dots, \alpha^{k-1}=\alpha^{2q-2}\}$ denote an enumeration of $(\F_{2^n})^*\cup\{\vec{0}\}=(\F_2)^n.$ Let $\{H_{j,n}\}_{j=0}^{2^n-1}$ denote the subgroups defined in equation~\eqref{eq:subgroups Hj}. Let $W_n$ be the $(2q=2^n)\times (2k=2(2^{n}-1))$ inclusion matrix whose rows are indexed by elements of $(\F_2)^n$ and whose columns are indexed by the subsets $\{H_{j,n}\}^{2^n-1}_{j=0}$ and their complements $\{H_{j,n}^c=(\F_2)^n\backslash H_{j,n}\}_{j=0}^{2^n-1},$ with the subgroups listed first and their respective complements indexing the column $k$ units later, i.e.
		\[(W_n)_{\vec{x},H_{j,n}}\;=\begin{cases}1,&\text{ if } \vec{x}\in H_{j,n}\\
		0,&\text{ if } \vec{x}\notin H_{j,n}\ ,\end{cases}
		\]
		and
		\[(W_n)_{\vec{x},(H_{j,n})^c}\;=\begin{cases}1,&\text{ if } \vec{x}\in H_{j,n}^c\\
		0,&\text{ if } \vec{x}\notin H_{j,n}^c\ \end{cases}
		\]
for $\vec{x}\in (\F_2)^n$. Then $W_n$ is a $2q\times 2k$ matrix that has full rank $2q.$
	\end{lemma}
	
	\begin{proof}  Lemma \ref{Key_intersection_ala_Nat} and the recursive construction of the subgroups $\{H_{j,n}\}_{j=0}^{2^n-1=k-1}$ shows that the matrix $W_n$ satisfies the conditions (R1) through (R5) of Lemma  \ref{lem:full rank general}, which we postpone to Section~\ref{subsection: general form of inclusion matrix} because of its algebraic nature.  It follows from Lemma \ref{lem:full rank general} that $W_n$ has rank $2q.$
	\end{proof}
	
	We shall also need the images of the groups $H_{j,n}$ inside the $\Z_2$-multispinal group $\Gf_n$. Explicitly, for $j=0,\dots,2^n-2$ let
\begin{equation}\label{eq: frakHj}
  \Hf_{j,n}=\iota_n(H_{j,n})\;\subset\; \Gf_n.
\end{equation}

	We  illustrate the above concepts in  the case of the first Grigorchuk group $G:=\mathfrak{G}_2$.

	\begin{example}\label{ex:Grigo}\emph{The first Grigorchuk group as a $\Z_2$-multispinal group.}

 Let $e,b,c,d, a$ denote the generators of this group, cf. \cite{Gri1}, see also \cite{Nekra-growth}, \cite{CEPSS}, \cite{Stei-Sza-Contr} and \cite{Y} for assertions on simplicity.   We use here the realisation of the first Grigorchuk group as a $\Z_2$-multispinal group associated with $H=\Z_2$ and $\F_4$, see comment after \cite[Theorem 7.10]{Stei-Sza-Contr}. As already mentioned, this case corresponds to the irreducible polynomial $f_2(x)=x^2+x+1$.
	
	Then with $b:=\alpha$ a zero of $f_2$, we have
	\begin{align*}
		b\vert_\bfo&=\varphi(\alpha)=\alpha^2=1+\alpha,\\
		b\vert_{\bfo^2}&=\varphi(1+\alpha)=\alpha(1+\alpha)=1.
	\end{align*}
	We have three order-two subgroups $H_{0,2}=\ker(\Tr)=\{0,1\}$, $H_{1,2}=\ker(\Tr\circ \varphi_2)=\{0,1+\alpha\}$ and $H_{2,2}=\ker(\Tr\circ \varphi_2^2)=\{0,\alpha\}$ in $\F_4$, with corresponding images in $\mathfrak{G}_2$ given by  $\Hf_{0,2}= \{e,d\}$, $ \Hf_{1,2}= \{e,c\}$ and $\Hf_{2,2}=\{e,b\}$. The matrix $W_2$ from Lemma~\ref{primitivepolynomial_lemma_over_F2} is given by
	\[
	W_2=\left(\begin{matrix}1&1&1&0&0&0\\
	0&0&1&1&1&0\\
	0&1&0&1&0&1\\
	1&0&0&0&1&1\\
	\end{matrix}\right);
	\]
A direct calculation shows that if we let
	\[
	T_2=\left(\begin{matrix}
1/3&-1/6&-1/6&1/3\\
1/3&-1/6&1/3&-1/6\\
	1/3&1/3&-1/6&-1/6\\
	-1/6&1/3&1/3&-1/6\\
	-1/6&1/3&-1/6&1/3\\
-1/6&-1/6&1/3&1/3\\
	\end{matrix}\right),
	\]
then $W_2T_2=I_4$, showing that $W_2$ has full rank equal to $4$. Alternatively, we can apply Lemma~\ref{primitivepolynomial_lemma_over_F2} with $q=2^{n-1}=2$.
\end{example}	
	
	\subsection{A general form of the inclusion matrix }\label{subsection: general form of inclusion matrix}

	We give here the  calculations of the rank of the inclusion matrix by  proving an extended result that applies to the matrix in Lemma~\ref{primitivepolynomial_lemma_over_F2}, besides additional matrices of similar form. Indeed, Lemma~\ref{primitivepolynomial_lemma_over_F2} follows from the next result by choosing $q=2^{n-1}$ and $k= 2^n-1$  for a fixed  integer $n\in\mathbb N^+.$
	
	\begin{lemma}\label{lem:full rank general}
		Let $q$ be a positive integer and let $k=2q-1.$
		Define a $(0,1)$-matrix $W=[W_{i,j}]$  of size $2q\times 2k$  by the prescription:
		\begin{enumerate}
			\item[(R1)]\label{eq:first row of W} $W_{1,j}=1,W_{1,j+k}=0 \text{ for all }1\leq j\leq k$.
			\item[(R2)]\label{eq:first row of W} There exist $j_1,\dots, j_{q-1}\in \{1,\dots, k\}$ so that
			\[
			W_{2,j}=\begin{cases}1&\text{ if }j=j_1,\dots, j_{q-1}\\0&\text{ if } j\in \{1,\dots, k\}\setminus \{j_1,\dots, j_{q-1}\}.\end{cases}
			\]
			\item[(R3)] For all $2\leq i\leq 2q,$ $0\leq \ell\leq 2q-i$ and $1\leq j\leq k$ we have
			\[W_{i+\ell,j  (\text{mod}\;k)}=W_{i,j+\ell (\text{mod}\;k)}.\]
			\item[(R4)]\label{eq:arbitrary row of W} For all $2\leq i\leq 2q$ and $1\leq j\leq k$ we have
			\[
			W_{i,j+k}=W_{i,j}+1\mod{2}.
			\]
			\item[(R5)] Let $0\;\leq\;r\;\leq k-1.$  Then  setting $K_r=\{j_1-r\;\text{mod}\;k,\;j_2-r\;\text{mod}\;k,\;\dots,\; j_{q-1}-r\;\text{mod}\;k \},$ we have
			$$|K_{r_1}\cap K_{r_2}|\;=\;\frac{q}{2}-1,\;\forall\;0\;\leq r_1<r_2\;\leq q-1.$$
		\end{enumerate}
		Then $W$ has rank $2q$.
	\end{lemma}
	
	Note that (R1)-(R5) imply that
	\begin{enumerate}
		\item[(R6)]For row 1, the matrix $W$ had exactly $k$ entries in the first $k$ columns that are equal to $1,$  and in every subsequent row $i=2,\dots, 2q$, the matrix $W$ has exactly $q-1$ entries in the first $k$ columns which are equal to $1$, i.e. $|\{j\mid 1\leq j\leq k, W_{i,j}=1\}|=q-1.$
		\item[(R7)] In the first row, there are $0$ entries in the second $k$ columns that are equal to $1,$ and in every subsequent row $i=2,\dots, 2q$, there are exactly $k-(q-1)=(2q-1)-q+1=q$ entries in the second $k$ columns which are equal to $1$, i.e. $|\{j\mid k+1\leq j\leq 2k, W_{i,j}=1\}|=q.$
		\item[(R8)]In every row $i=1,\dots, 2q$, there are exactly $q-1+k-(q-1)=k=2q-1$ entries in all the $2k$ columns which are equal to $1$, i.e. $|\{j\mid 1\leq j\leq 2k, W_{i,j}=1\}|=k.$
		\item[(R9)] In every column $j=1,\dots, 2k$, there are $k-q+1=q$ entries which are equal to $1$, i.e. if $j$ is fixed, $|\{i\mid 1\leq i\leq 2q, W_{i,j}=1\}|=q.$  Note that this is true because for fixed $r$ with $0\leq r\leq k-1,\;W_{1,1+r}=1$ by construction, and for $2\leq i\leq 2q,\;W_{i,1+r}=1$ if and only if $i\in K_{r}.$  So for fixed $j$ with $1\leq j\leq k,\;|\{1\leq i\leq 2q:\;W_{i,j}=1\}|=1+(q-1)=q.$   For $j\in \{k+1,k+2,\dots, 2k\},$ by construction, $W_{1,j}=0,$ and by mirroring, for $2\leq i\leq 2q,$  $W_{i,j}=1$ if and only if $i\in \{1,2,\dots,k\}\backslash K_{j-(k+1)}.$ So for fixed $j$ with $k+1\leq j\leq 2k$ we have
\[
\;|\{1\leq i\leq  2q:\;W_{i,j}=1\}|=|\{1,2,\dots,k\}\backslash K_{j-(k+1)}|=k-(q-1)=q.
\]   It then follows that for $1\leq j\leq k,$ the $j^{\text{th}}$ column has exactly $q-1$ entries equal to $1$ in rows $2$ through $2q$  (this is because all of the entries in row $1$ in columns $j=1$ through $k$ are equal to $1$).
		
	\end{enumerate}
	
	We remark that examples of matrices $W$ as in the lemma will arise from the simplicity analysis of Steinberg algebras associated to the multispinal groups from \cite{Stei-Sza-Contr}, see
section~\ref{section:simplicity}.

	\begin{proof}
		After we complete our direct proof, we will remark on how the result also follows from combinatorial design theory.
		
		To prove $W$ has full rank, it suffices to find a $2k\times 2q$ matrix $T$ so that $WT=I$, the identity of size $2q\times 2q$. We define $T$ as follows:
		\[
		T_{j,i}=\begin{cases}\frac {1}{k},&\text{ if } W_{i,j}=1\\
		-\frac {q-1}{k(k-q+1)},&\text{ if } W_{i,j}=0,\end{cases}
		\]
		and
		\[
		T_{j+k,i}=\begin{cases}\frac {1}{k},&\text{ if } W_{i,j+k}=1\\
		-\frac {q-1}{k(k-q+1)},&\text{ if } W_{i,j+k}=0,\end{cases}
		\]
		for all $1\leq i\leq 2q, 1\leq j\leq k$. By our hypotheses, for each fixed $j\in \{1,\dots, k\}$ we have the following:
		\begin{enumerate}
			\item

			For every $1\leq i\leq 2q$, there are $k=2q-1$ columns with label $\{j_{i,1}<j_{i,2}<\cdots j_{i,k}\}$ where $W_{i,j_{i,\ell}}=1$, corresponding to row entries \lq \lq $j_{i,\ell},i$" in $T$ so that $T_{j_{i,\ell},i}=1/(2q-1)$, and there are $k-q+1=q$ columns with label \lq \lq $j$" where $W_{i,j}=0$, leading to rows in $T$ so that  $T_{j,i}=-(q-1)[(2q-1)(q)]^{-1}$.
			\item For every $1\leq i\leq 2q$, there are $k-q+1=q$ columns with label $``j+k"$ where $W_{i,j+k}=W_{i,j}+1 (\text{mod}\;2) =1$, corresponding to rows $``j+k"$ in $T$ where $T_{j+k,i}=1/k$, and there are $q-1$ columns with label $``j+k"$ where $W_{i,j+k}=W_{i,j}+ 1 (\text{mod}\;2)=0$, leading to corresponding rows in $T$ where $T_{j+k,i}=-(q-1)[k(k-q+1)]^{-1}$.
		\end{enumerate}
		We then have
		\begin{align*}
			(WT)_{i,i}
			&=\sum_{j=1}^{k} W_{i,j}T_{j,i}+\sum_{j=1}^k W_{i,j+k}T_{j+k,i}\\
			&=(q-1)\cdot 1\cdot \frac 1k +(k-q+1)\cdot 0\cdot (-\frac {q-1}{k(k-q+1)})\\
			&+(q-1)\cdot 0\cdot (-\frac {q-1}{k(k-q+1)}) +(k-q+1)\cdot 1\cdot \frac 1k=1\\
		\end{align*}
		for all $1\leq i\leq 2q$.
		
		For the purposes of the next calculation, we denote
\[
a=\frac 1k \text{ and }b=-\frac{q-1}{k(k-q+1)}=-\frac{q-1}{(2q-1)\cdot q}.
\]
Let $i,l\in \{1,\dots, 2q\}$ with $i\neq l$. Then
		\begin{align*}
			(WT)_{i,l}
			&=\sum_{\{j\mid 1\leq j\leq k, W_{i,j}=1\}}T_{j,l}+\sum_{\{j\mid 1\leq j\leq k, W_{i,j+k}=1\}}T_{j+k,l}\\
			&=\sum_{\{j\mid 1\leq j\leq k, W_{i,j}=1, T_{j,l}=a\}}a+\sum_{\{j\mid 1\leq j\leq k, W_{i,j}=1, T_{j,l}=b\}}b\\
			&+\sum_{\{j\mid 1\leq j\leq k, W_{i,j+k}=1, T_{j+k,l}=a\}}a+\sum_{\{j\mid 1\leq j\leq k, W_{i,j+k}=1, T_{j+k,l}=b\}}b\\
			&=\sum_{\{j\mid 1\leq j\leq k, W_{i,j}=1, T_{j,l}=a\}}a+0+0+\sum_{\{j\mid 1\leq j\leq k, W_{i,j}=1, T_{j,l}=b\}}b\\
		\end{align*}

		There are now three cases to consider:  we first consider $i=1<l\leq 2q.$
		Then
		\begin{align*}
			(WT)_{1,l}
			&=\sum_{\{j\mid 1\leq j\leq k\}}T_{j,l}\\
			&=\sum_{\{j\mid 1\leq j\leq k, T_{j,l}=a\}}a+\sum_{\{j\mid 1\leq j\leq k,  T_{j,l}=b\}}b\\
			&=(q-1)\cdot a+ q \cdot b\\
			&=(q-1)\cdot \frac 1k + 	q \cdot (-\frac {q-1}{k(k-q+1)})=0.
		\end{align*}
		So $(WT)_{1,l}=0,\;\;2\leq\;l\;\leq 2q.$

		We now consider the case where $i,l\in \{2,\cdots, 2q\},\;i\not= l.$
		Before proceeding, we define four sets that are of interest in our calculations, and make some remarks about their cardinalities.  For $i, l\in \{2,3,\cdots,2q \}$ with $i\not= l,$ let
		$$F_{i,l} = \{j: 1\leq j\leq k:  W_{i,j}=1 \;\text{and}\; W_{l,j}=1\},\;\;
		G_{i,l}=\{j: 1\leq j\leq k:  W_{i,j}=1 \;\text{and}\; W_{l,j}=0\},$$
		$$H_{i,l}=\{j: 1\leq j\leq k:  W_{i,j}=0 \;\text{and}\; W_{l,j}=0\},\;\text{and}\;
		I_{i,l}= \{j: 1\leq j \leq k:  W_{i,j}=0 \;\text{and}\; W_{l,j}=1\}.$$
		We have, for $2\leq i \not= l\leq 2q,$ using $(R7)$ and $(R9),$
		$$|F_{i,l}|+|G_{i,l}|=(q-1),\;|H_{i,l}|+|I_{i,l}|= [k-(q-1)]=q,$$
		$$|F_{i,l}|+|I_{i,l}|=(q-1),\;\text{and}\;|G_{i,l}|+|H_{i,l}|=k-(q-1)]=q.$$
		From two of the equations, we see that $|G_{i,l}|=|I_{i,l}|.$
		From Condition (R5), we see that $|F_{i,l}|= |K_{i-1}\cap K_{l-1}|=\frac{q}{2}-1.$  It follows from our cardinality equations again that $|G_{i,l}|=|I_{i,l}|=|H_{i,l}|=\frac{q}{2}.$
		
		Then:
		\begin{align*}
			(WT)_{i,l}
			&=\sum_{\{j\mid 1\leq j\leq k, W_{i,j}=1\}}T_{j,l}+\sum_{\{j\mid 1\leq j\leq k, W_{i,j+k}=1\}}T_{j+k,l}\\
			&=\sum_{\{j\mid 1\leq j\leq k, W_{i,j}=1, W_{l,j}=1\}}a+\sum_{\{j\mid 1\leq j\leq k, W_{i,j}=1, W_{l,j}=0\}}b\\
			&+\sum_{\{j\mid 1\leq j\leq k, W_{i,j}=0, W_{l,j}=0\}}a+\sum_{\{j\mid 1\leq j\leq k, W_{i,j}=0, W_{l,j}=1\}}b\\
			&=|F_{i,l}|\cdot a+|I_{i,l}|\cdot b+|H_{i,l}|\cdot a+ |G_{i,l}|\cdot b\\
			&=(\frac{q}{2}-1)a+ (\frac{q}{2}) b +  (\frac{q}{2}) a+   (\frac{q}{2}) b\\
			&= (q-1)\cdot a + q\cdot b \\
			&=(q-1)\cdot \frac 1k + 	q \cdot (-\frac {q-1}{k(k-q+1)})=0.
		\end{align*}
		So $(WT)_{i,l}=0,$ for  $i,l\in \{2,\cdots, 2q\},\;i\not= l.$

		Finally, for $2\leq i\leq 2q,$ and $l=1$ we have:
		\begin{align*}
			(WT)_{i,1}
			&=\sum_{\{j\mid 1\leq j\leq k\}}W_{i,j}T_{j,1}+\sum_{\{j\mid 1\leq j\leq k\}}W_{i,j+k}T_{j+k,1}\\
			&=\sum_{\{j\mid 1\leq j\leq k\}}W_{i,j}\cdot a+\sum_{\{j\mid 1\leq j\leq k\}}W_{i,j+k}\cdot b\\
			&= (q-1)\cdot a + q\cdot b \\
			&=(q-1)\cdot \frac 1k + 	q \cdot (-\frac{q-1}{k(k-q+1)})=0.
		\end{align*}
		
		These are all the results  we need to conclude $WT=I,$ which finishes the proof of the Lemma.
	\end{proof}

	\begin{remark}
		\label{design-Nat}
		We note that if we set $A$ to be the $(2q-1)\times k\;=\;(2q-1)\times (2q-1)$ square matrix whose  $(2q-1)$ rows are exactly the $2^{\text{nd}}$ through $2q^{\text{th}}$ rows of $W,$ and whose  $(2q-1)=k$ columns are the first $k$ columns of $W,$ then a little linear algebra shows that $W$ has full rank $2q$ if and only if $A$ has full rank $2q-1=k.$  But Hypotheses (R1) through (R5) and deductions (R6) through (R9) show that $A$ is the incidence matrix of a $2-(k=2q-1, q-1, \frac{q}{2}-1)$-design in the sense of Assmus and Key \cite{Assmus}, Definition 1.2.1, with underlying set ${\mathcal P}\;=\;\{1,2,\dots,2q-1=k\}.$  We can then deduce that the matrix $A$ has full rank $2q-1=k$ from \cite{Assmus}, Theorem 1.4.1.
		\end{remark}

	\section{$\Z_2$-Multispinal groupoids and simplicity of the Steinberg algebra}\label{section:simplicity}
	
	We now consider the groupoids $\G_n$ associated to the
	$\Z_2$-multispinal groups $\mathfrak{G}_n$ of Section \ref{def:multispinal}. The construction of these groupoids  is
	described in the general case of self-similar actions
	in Section \ref{sec:back-Gpds}. The goal of this section is to show that the Steinberg algebra $\AA_\mathbb{K}(\mathcal{G}_n)$ of $\mathcal{G}_n$ over a field $\mathbb{K}$ of characteristic zero is simple.

	We will make use of a family of compact open bisections $U_{m}(z_g)$ (denoted by $U_{g,m}$ in \cite{CEPSS}) where the points $z_g\in \mathcal{G}_n$,  for $g\in \mathfrak{G}_n$, are defined by:
	\[
	z_g := [(\emptyset, g, \emptyset), \bfo^\infty ],
	\]
	see also \cite[Section 5]{CEPSS}. Explicitly $U_{m}(z_g)$  is a neighborhood of $z_g$ in $\mathcal{G}_n$, where we recall that we use $\Theta$ to denote bisections, and is defined as
	\[
U_{m}(z_g):=\Theta\bigl( (\varnothing,g,\varnothing), C(\bfo^m)\bigr).
	\]
		 By construction, $U_\ell(z_g)\subseteq U_m(z_g)$ for $\ell\geq m$. We will use both notations, $U_{m}(z_g)$ and $U_{g,m}$, depending on whether we highlight the connection with elements in $\Gf_n$ or points in $\mathcal{G}_n$.
Note that $\mathcal{G}_n$ is countably ample with the choice of basis around points $z_g, g\in \mathcal{N}_{0,n}$ given by $\{U_m(z_g)\}_{m\in \N}$.

	The following lemma details the structure of the support sets of
	bisections in terms of the sets $U_\ell(z_g)$.

	\begin{lemma}\label{lem:primitive polynomial over Z two} Let $f_n(x)$ be a primitive polynomial of degree $n\geq 2$ over $\Z_2$,  identify $\Z_2[x]/\langle f_n(x)\rangle$ with the field $\F_{2^n}$, and let $\alpha$ be a zero of $f_n(x)$ and a generator of the cyclic group $\F_{2^n}^*$ on $2^n-1$ elements. Form the $\Z_2$-multispinal group $\mathfrak{G}_n$ corresponding to $\Z_2$ and $\F_{2^n}$ with its automorphism implemented by multiplication by $\alpha$. Let $m\geq 1$ and set $\mathcal{F}_m:=\{U_m(z_g)\mid g\in \mathcal{N}_{0,n}\}$, where we recall that $\mathcal{N}=\mathcal{N}_{0,n}\sqcup\{a\}$ denotes the nucleus of $\mathfrak{G}_n$. With notation as in \eqref{eq: frakHj}, let $K\in \{\Hf_{j,n}, \Hf_{j,n}^c\mid j=0,\dots,2^n-2\}$, and let $\JJ=\{U_m(z_g)\mid g\in K\}\subset \mathcal{F}_m$. Then with notation as in  Equation (\ref{eq:M disjointified})
 we have the following structure result for the compact open bisections $M_\JJ$:
		\[
		M_\JJ=\bigcap_{g\in K} U_m(z_g).
		\]
	\end{lemma}
	
	\begin{proof}
Let $g_1,g_2$ be distinct in $\mathcal{N}_{0,n}$. We claim that $U_m(z_{g_1})\cap U_m(z_{g_2})$  is non-empty. To  have  $z\in U_m(z_{g_1})\cap U_m(z_{g_2})$ requires that we can find $w\in C(\bfo^m)$ so that $
		z=[(\varnothing,g_1,\varnothing),w]=[(\varnothing,g_2,\varnothing),w].$ Thus, we need \[
		(\varnothing,g_1,\varnothing)(\bfo^m\mu,e,\bfo^m\mu)=(\varnothing,g_2,\varnothing)(\bfo^m\mu,e,\bfo^m\mu)
		\]for some $\mu\in X^*$ (where $X=\{ \bfz,\bfo\}$) or, equivalently,
		\begin{align}
		g_1\vert_{\bfo^m}\cdot\mu=g_2\vert_{\bfo^m}\cdot\mu &\text{ and }(g_1g_2^{-1})\vert_{\bfo^m\mu}=e.
		\label{eq:nonempty intersection of two}
		\end{align}
		Denote $x_1:=\iota_n^{-1}(g_1)$, $x_2:=\iota_n^{-1}(g_2)$ in $\F_{2^n}$, where $x_1\neq x_2$.  Let $0\leq j<2^n-2$ be such that $x_1-x_2\in H_{j,n}$ (by (R6) there are $q-1=2^{n-1}-1$ such indices). Then, by the definition of  restrictions in \eqref{eq:restriction to one} we obtain
		\[
		\bigl((g_1g_2^{-1})\vert_{\bfo^j}\bigr)\vert_\bfz=e.
		\]
		Now let $m'\in \N$ be such that $m+m'\equiv j\mod(2^n-1)$. Then the identities in \eqref{eq:nonempty intersection of two} are satisfied
with $\mu:=\bfo^{m'}\bfz$, so $U_m(z_{g_1})\cap U_m(z_{g_2})$ is non-empty and consists of classes
		\begin{equation}\label{eq:m+m'}
		\{[(\varnothing,g_i,\varnothing),\bfo^{m+m'}\bfz  w']\mid w'\in X^\omega,m+m'\equiv j\mod(2^n-1)\},
		\end{equation}
		with $i=1,2$.
		
		Now fix $0\leq j<2^n-2$, let $K=\Hf_{j,n}$ and $\JJ=\{U_m(z_g)\mid g\in K\}\subset \mathcal{F}_m$.  With the notation of \eqref{eq:M disjointified}, consider
		\[
		M_\JJ=\bigcap_{g\in K} U_m(z_g)\setminus\bigcup_{h\in K^c}U_m(z_h).
		\]
		To prove the assertion of the lemma for this choice of $K$ it suffices to establish that
		\begin{equation}\label{eq:K intersect but not with K complement}
		\bigcap_{g\in K} U_m(z_g)\cap U_m(z_h)=\emptyset \text{ for all }h\in K^c.
		\end{equation}
		Note that $g_1g_2^{-1}\in \Hf_{j,n}$ for every pair $g_1,g_2\in K$, and so the set $\bigcap_{g\in K} U_m(z_g)$ is nonempty since for some  $g\in K$ it consists of the classes
		\[
		[(\varnothing, g,\varnothing),\bfo^{(2^n-1)q+j}\bfz\mu'], q\in \N\text{ so that }(2^n-1)q+j\geq m,\mu'\in X^\omega.
		\]
		Let $h\in K^c$. Then  $(\Tr\circ \varphi_n^j)(\iota_n^{-1}(h))=1$, thus $gh^{-1}\notin \Hf_{j,n}$ for all $g\in K$ and so no extension $\bfo^{m'}\bfo^m$ of $\bfo^m$ can start with $\bfo^j$.  In particular, no classes as in equation \eqref{eq:m+m'} exist for $h\in K^c$,  which implies our claim in \eqref{eq:K intersect but not with K complement}.
		
		A similar argument proves the assertion of the lemma for $K=\Hf_{j,n}^c$:
		first note that $\bigcap_{g\in K} U_m(z_g)$ is nonempty because it contains $[(\varnothing,g,\varnothing),\bfo^{(2^n-1)q+j}\bfz\mu']$ for $q\in \N$ with $(2^n-1)q+j\geq m$ and $\mu'\in X^\omega$ for some (and thus all) $g\in K$ (since $g_1g_2^{-1}\in \Hf_{j,n}$ when  $g_1,g_2$ are distinct elements of $\Hf_{j,n}^c$). Moreover, if $h\in K^c=\Hf_{j,n}$, then $(\Tr\circ \varphi_n^j)(\iota_n^{-1}(h))=0$ and so no classes as in \eqref{eq:m+m'} exist for this choice of $h$ and all $g\in K$.
	\end{proof}

\subsection{Simplicity of the Steinberg algebra} To prove simplicity of the
Steinberg algebra of a $\Z_2$-multispinal groupoid over any field $\mathbb{K}$ of characteristic zero  will require first a characterization of singular functions. We will address this point next.
The following lemma is in fact the analogue of \cite[Lemma 5.25]{CEPSS} in the case of an arbitrary degree $n$ primitive polynomial over $\Z_2$ and it characterizes singular elements of the Steinberg algebra. By $1_U$ we denote the characteristic function of a compact open set $U\subset \mathcal{G}_n.$
	
	\begin{lemma}\label{lem:full rank} Let $\mathbb{K}$ be a field with $\chara(\K)=0$ and $\mathcal{G}_n$  the ample groupoid associated to the multispinal group $\mathfrak{G}_n$ from a primitive polynomial of degree $n\geq 2$ over $\Z_2$. Suppose that $f\in A_{\mathbb{K}}(\mathcal{G}_n)$ is such that
		\begin{equation}\label{eq:f as combination of Um from nucleus}
		f=\sum_{g\in \NN_{0,n}}c_g1_{U_m(z_g)} \text{ for }c_g\in \mathbb{K} \text{ and }m\in \N.
		\end{equation}
		Then $f$ belongs to the ideal of singular elements $S_\K(\mathcal{G}_n)$ precisely when it is the zero function.
	
	\end{lemma}
	
	\begin{proof} Let $f$ be given as in \eqref{eq:f as combination of Um from nucleus}. Assume that $f\in S_\K(\mathcal{G}_n)$. By equation \eqref{eq:function disjointification}, the support of $f$ can be expressed as a disjoint union of sets $M_\JJ$ where $\JJ$ runs over the non-empty subsets of $\mathcal{F}_m=\{U_m(z_g)\mid g\in \mathcal{N}_{0,n}\}.$

We shall show that it suffices to restrict attention to the family of sets $M_\JJ$ corresponding to the choices $\JJ=\{U_m(z_g)\mid g\in K\}$ where $K$ is $\Hf_{j,n}$ or $\Hf_{j,n}^c$ for some $j$ with $0\leq j\leq 2^n-2$. Indeed, using only these sets in the disjointified decomposition of the support of $f$ we shall show that all $c_g$ where $g\in \NN_{0,n}$ vanish.

Let $K$ be  $\Hf_{j,n}$ or $\Hf_{j,n}^c$ for some $j$ with $0\leq j\leq 2^n-2$,  and let $\JJ=\{U_m(z_g)\mid g\in K\}$. By Lemma~\ref{lem:primitive polynomial over Z two} we know that
		\[
		M_\JJ=\bigcap_{g\in K} U_m(z_g),
		\]
		thus $M_\JJ$ has non-empty interior, being a non-empty finite intersection of open sets. The assumption on  $f$ implies that when evaluated at $x\in M_\JJ$, the function vanishes, so
		\[
		f\vert_{M_\JJ}(x)=\sum_{g\in K}c_g=0.
		\]
		Letting $K$ and $j$ vary, this gives rise to a system of $2k=2(2^{n}-1)$ equations
\[
\sum_{g\in K}c_g=0,
\]
with one pair of equations for each choice of $\Hf_{j,n}$ and $\Hf_{j,n}^c$ and with $2q=2\cdot 2^{n-1}$ unknowns, corresponding to $c_g$ where $g\in \mathcal{N}_{0,n}$. The underlying matrix $\tilde{W}_n$ of size $2(2^n-1)\times 2^n$ is the transpose of the matrix $W_n$ from Lemma~\ref{primitivepolynomial_lemma_over_F2}, hence is of full rank $2^n$. It follows that the system admits only the trivial solution $c_g=0$, $g\in \mathcal{N}_{0,n}$, which  implies the claim $f\equiv 0.$
	\end{proof}
	
	\begin{remark}\label{rmk:lemma 5.17 of 5 authors} Let $\mathcal{G}_n$ be the groupoid associated with a $\Z_2$-multispinal group $\mathfrak{G}_n$. 	Suppose that $D=\Theta((\gamma,g,\mu), C(\mu \eta))$ is a compact open bisection in $\mathcal{G}_n$. If
		$[(\bfo^s, h, \bfo^s),\bfo^\infty]\in \overline{D}$ where $h\in \mathcal{N}_{0,n}$  and $s\in \N$, then by \cite[Lemma 5.17]{CEPSS} we have that:
		\begin{enumerate}
			\item[(a)] There exists $\eta'\in X^*$ so that $\bfo^s=\mu\eta\eta'=\gamma(g\cdot(\eta\eta'))$.
			\item[(b)] For all $N\geq 1$ there exists $w\in X^*$ satisfying that
			\[
			g\vert_{\eta\eta'}\cdot (\bfo^Nw)=h\cdot(\bfo^Nw)\text{ and }(g\vert_{\eta\eta'})\vert_{\bfo^Nw}=h\vert_{\bfo^Nw}.
			\]
		\end{enumerate}
		Since $h$ fixes both $\bfo$ and $\bfz$, the first condition in (ii) becomes $g\vert_{\eta\eta'}\cdot (\bfo^Nw)=\bfo^Nw$.
	\end{remark}
	
	The next series of lemmas will eventually prove the simplicity of $\AA_\mathbb{K}(\mathcal{G}_n)$  when $\mathbb{K}$ has characteristic zero. The next result is the analogue of \cite[Lemma 5.27]{CEPSS} in the case of an arbitrary primitive polynomial over $\Z_2$.
	
	\begin{lemma}\label{lem:one in D bar means all in D bar}
		Let $\mathbb{K}$ be a field with $\chara(\K)=0$ and $\mathcal{G}_n$  the ample groupoid associated to the $\Z_2$-multispinal group $\mathfrak{G}_n$ from a primitive polynomial of degree $n\geq 2$ over $\Z_2$. Let $D=\Theta((\gamma,g,\mu), C(\mu \eta))$   be a compact open bisection, and recall that $\mathcal{N}_{0,n}$ is the image of $\F_{2^n}$ under $\iota_n$ in $\mathfrak{G}_n$. Then the following are equivalent:
		\begin{enumerate}
			\item For some $h\in \mathcal{N}_{0,n}$ we have $[(\varnothing, h, \varnothing),\bfo^\infty]\in \bar{D}$.
			\item For all $h\in \mathcal{N}_{0,n}$ we have $[(\varnothing, h, \varnothing),\bfo^\infty]\in \bar{D}$.
		\end{enumerate}
\end{lemma}
	
	\begin{proof} Assume that $[(\varnothing, h, \varnothing),\bfo^\infty]\in \bar{D}$ for some $h\in \mathcal{N}_{0,n}$. Since $2^n-1$ is the period of restriction to $\bfo$, we have that $
		[(\varnothing, h, \varnothing),\bfo^\infty]=[(\bfo^s,h, \bfo^s),\bfo^\infty]
		$
		for all $s\in (2^n-1)\Z_{>0}$.
		
		By definition of the nucleus, there exists $m_0\in \N$ such that $g\vert_{v}\in \mathcal{N}_{0,n}$ whenever  $v\in X^\omega$ satisfies $|v|\geq m_0$.  Choose $m_1\in \N$ such that
		\[
		s:=m_0+m_1+|\mu\eta|=(2^n-1)s', s'\in \N.
		\]
		Let $r:=m_0+m_1$ and $\eta':=\bfo^r$. The assumption on $h$ implies that condition (a) in Remark~\ref{rmk:lemma 5.17 of 5 authors} holds, thus  $\mu=\bf1^{|\mu|}$, $\eta=\bfo^{|\eta|}$, $\eta'=\bfo^{r}$, $\gamma=\mu$ and
		\begin{equation}\label{eq:g fixes something}
		\bfo^s=\bfo^{r+|\mu\eta|}=\bfo^{|\mu|}(g\cdot(\bfo^{r+|\eta|})),\text{ in particular }\bfo^{r+|\eta|}=g\cdot(\bfo^{r+|\eta|});
		\end{equation}
		condition (b) likewise holds, thus for a given $N\in \N$ there exists $w\in X^*$ with
 		\begin{equation}\label{eq:g fixes 1 N word}
		(g\vert_{\bfo^{r+|\eta|}})\cdot \bfo^Nw=\bfo^Nw.
		\end{equation}
		This last equality in conjunction with the definition of $a$ in $\mathfrak{G}_n$ and the choice of $r$ imply that $g\vert_{\bfo^{|\eta|+r}}\in \mathcal{N}_{0,n}$. In particular, these considerations show that assuming condition (1) of the lemma,  the set $D$ has  form
\begin{equation}\label{eq: general form of D}
D=\Theta((\bfo^L,g,\bfo^L),C(\bfo^L\bfo^p)), \text{ with }g\vert_{\bfo^{r+p}}\in \mathcal{N}_{0,n}, p+r+L\in (2^n-1)\Z_{>0}
\end{equation}
where $L,p,r$ are positive integers (here $L=|\mu|, p=|\eta|$).

Let $c\in \{0,1,\dots, 2^n-1\}$ be such that $g\vert_{\bfo^{|\eta|+r}}=\iota_n(\alpha^c)$. Take $h'\in \mathcal{N}_{0,n}$ distinct from $h$.  To show that $[(\varnothing, h', \varnothing),\bfo^\infty]\in \bar{D}$, we must verify the conditions in \cite[Lemma 5.17]{CEPSS}. We can run the previous argument backwards with $\eta'=\bfo^r\in X^*$.  Then condition (a) is already satisfied in the form  $\bfo^s=\bfo^{(2^n-1)s'}=\bfo^{m_0+m_1+|\eta\mu|}=\mu\eta\eta'$, and for condition  (b) we let  $N\in \N$, $N\geq 1$ and need to find $w\in X^*$ such that
		\begin{equation}\label{eq:condition two of lemma 5.17}
		\bigl((h')^{-1}\iota_n(\alpha^c)\bigr)\cdot \bfo^Nw=\bfo^Nw\text{ and }\iota(\alpha^c)\vert_{\bfo^Nw}=h'\vert_{\bfo^Nw}.
		\end{equation}
		If $h'$ and $\iota_n(\alpha^c)$ are equal there is nothing to prove, so assume that $h'=\iota_n(\alpha^l)$ for some $l=0,1,\dots, 2^n-1$, $l\neq c$. The first condition in \eqref{eq:condition two of lemma 5.17} is satisfied since $\iota_n(\alpha^l)^{-1}\iota_n(_n\alpha^c)$ is in $\mathcal{N}_{0,n}$ and thus fixes both $\bfo$ and $\bfz$. The second condition becomes
		\[
		\iota_n(\varphi^N(\alpha^c-\alpha^l))\vert_w=e,
		\]
		which means we can take $w=\bfo^{N'}\bfz$ for some $N'\in \N$ so that $\alpha^c-\alpha^l\in \ker(\Tr\circ \varphi_n^{N+N'})$, similar to the proof of the assertion $U_N(z_{\iota_n(\alpha^c)})\cap U_N(z_{\iota_n(\alpha^l)})\neq \emptyset$ in Lemma \ref{lem:primitive polynomial over Z two}.
	\end{proof}
	
	The next result is an analog of \cite[Lemma 5.28]{CEPSS}, whose proof we follow.

	\begin{lemma}\label{lemma:restriction}
		Let $D=\Theta((\gamma,g,\mu), C(\mu \eta))$ be a compact open bisection in the groupoid associated to $\mathfrak{G}_n$. If  $[(\varnothing, e, \varnothing),\bfo^\infty]\in \bar{D}$, then there are $m_D\in \N$ and $g_D\in \mathcal{N}_{0,n}$ such that
\[
		D\cap (\bigcup_{g\in \mathcal{N}_{0,n}}U_{m_D}(z_g))=U_{m_D}(z_{g_D}).
		\]
	\end{lemma}
	
	\begin{proof}
		Since $e\in \mathcal{N}_{0,n}$, we may assume that $D$ is of the form given in equation \eqref{eq: general form of D}.
		Let $t=r+p$. Then
		\[
		\Theta((\bfo^{L+t},g\vert_{\bfo^t},\bfo^{L+t}),C(\bfo^{L+t}))\subset D.
		\]
			Supposing now that $g\vert_{\bfo^{t}}=\iota_n(\alpha^{c})$, choose $l=0,1,\dots,2^n-1$ so that
		\[
		(L+t)+l\equiv c \mod{(2^n-1)}.
		\]
		Now set $g_D:=\iota_n(\alpha^l)\in \mathcal{N}_{0,n}$ and $m_D:=L+t$.  Exactly as in the previous inclusion above, we have that
		\[
		U_{m_D}(z_{g_D})\subset D \text{ since }g_D\vert_{\bfo^{L+t}}=g\vert_{\bfo^{t}}.
		\]
To conclude the proof of the lemma,
		it remains to prove that for every $g'\in \mathcal{N}_{0,n}$ with $g'\neq g_D$ we have the inclusion
\[
D\cap U_{L+t}(z_{g'})\subset U_{L+t}(z_{g_D}).
\]
Writing $g'=\iota_n(\alpha^s)$ where $\alpha^s\neq \alpha^l$, this reduces to showing that the nonzero element
$\varphi_n^{L+t}(\alpha^s)-\varphi_n^t(\alpha^c)$ is in some subgroup $\Hf_{j,n}$, which is true by the definition in \eqref{eq:subgroups Hj}.
	\end{proof}

\begin{lemma}\label{lem:f restricts to combination of characteristic fcts of Um} If $f\in \mathcal{A}_\K(\mathcal{G}_n)$, then there is a non-empty open set $U_m$ so that $f\vert_{U_m}$ is a linear combination, with coefficients in $\mathbb{K}$, of characteristic functions $1_{C}$, with $C\in \mathcal{F}_m$ for some $m\geq 1$.
\end{lemma}

\begin{proof} The proof is similar to the first half  of \cite[Lemma 5.29]{CEPSS}, and we adopt some of their notation. Let $f\in \mathcal{A}_\K(\mathcal{G}_n)$, thus  $f=\sum_{B\in \FF}a_B1_B$ for a finite set $\FF$ of compact open bisections and $a_B\in \K$. Let $U_m=\bigcup_{z_f\in \mathcal{Z}}U_{m}(z_f)$. We will show that there is $m\in \N$ such that  $ f(x)=0$ for $x\notin U_m$ and
		\[
		f\vert_{U_m}\in \operatorname{span}_\K \{1_N\mid N\in \mathcal{F}_m\}.
		\]
		
		Let $\mathcal{Z}:=\{z_g\mid g\in \mathcal{N}_{0,n}\}$. Suppose that $B\in \FF$ is so that $\overline{B}\cap \mathcal{Z}=\emptyset$. Then for every $g\in \mathcal{N}_{0,n}$ there is a positive integer $m_{g,B}$ such that $B\cap U_{m_{g,B}}(z_g)=\emptyset$. Consequently, $B\cap U_l(z_g)=\emptyset$ for  all $l\geq m_{z,B}$. With $m_B=\operatorname{max}_{g\in \mathcal{N}_{0,n}}m_{g,B}$, we have that $B\cap\bigcup_{z\in\mathcal{Z}} U_{m_B}(z_g)=\emptyset$ for all $l\geq \operatorname{max}_B\{m_B\}$. Let $U_l=\bigcup U_l(z_g)$. Then
		\[
		\bigl(\sum_{\{B\mid \overline{B}\cap \mathcal{Z}=\emptyset\}}a_B1_B\bigr)\vert_{U_{l}}=\sum_{B} a_B1_{B\cap U_{l}}=0.
		\]
		
		Now suppose that $D\in \FF$ satisfies $\overline{D}\cap \mathcal{Z}\neq \emptyset$. By Lemma~\ref{lem:one in D bar means all in D bar} we may assume that $z_e\in \overline{D}$. Then by Lemma~\ref{lemma:restriction}, there are $f_D\in \mathcal{N}_{0,n}$ and $m_D\in \N$ such that $D\cap (\bigcup_{f\in \mathcal{N}_{0,n}}U_{m_D}(z_f))=U_{f_D, m_D}$. Then also
\[
D\cap \bigcup_f U_j(z_f)=U_{f_D,j}=U_j(z_{f_D}).
\]
for all $j\geq m_D$.  Set $m=\operatorname{max}\{m_B,m_D\}$. Let $\FF_{g}$ denote the subset of all $D\in \FF$  with the property that $D\cap \bigcup_{z_f\in \mathcal{Z}}U_{m}(z_f)=U_m(z_{g})$.  It follows that
		\[
		\bigl(\sum_{\{D\mid \overline{D}\cap\mathcal{Z}\neq\emptyset\}}a_D1_D\bigr)\vert_{U_m}=\sum_{D} a_D1_{U_m(z_D)}=\sum_{g\in\mathcal{N}_{0,n}}(\sum_{D\in \FF_g}a_D) 1_{U_m(z_g)},
		\]
		which implies the lemma.
	\end{proof}

The following is an analog of \cite[Lemma 5.29]{CEPSS}.
	
	\begin{lemma}\label{lem:f nonzero on ze is singular} 	Let $\mathbb{K}$ be a field with $\chara(\K)=0$ and $\mathcal{G}_n$  the ample groupoid associated to the $\Z_2$-multispinal group $\mathfrak{G}_n$ from a primitive polynomial of degree $n\geq 2$ over $\Z_2$. Then if  $f\in \mathcal{A}_{\K}(\mathcal{G}_n)$ with $f(z_e) \not= 0$, we have that  $supp(f)$ has nonempty interior. Thus, $f\notin \mathcal{S}_\K(\mathcal{G}_n)$.
\end{lemma}

	\begin{proof} Take $f\in \mathcal{A}_{\K}(\mathcal{G}_n)$ with $f(z_e) \not= 0$. By Lemma~\ref{lem:f restricts to combination of characteristic fcts of Um}, there is an open set $U_m$ such that $f\vert_{U_m}$ is a linear combination as in equation \eqref{eq:f as combination of Um from nucleus}. Since $z_e\in U_t$ for every $t\in \N$, we infer that $f\vert_{U_m}(z_e)\neq 0$, so the support of $f\vert_{U_m}$ has non-empty interior by Lemma~\ref{lem:full rank}. But any open set inside $U_m$ is also in the interior of the support of $f$, which proves the assertion of the Lemma.
		\end{proof}

The next result  generalizes \cite[Lemma 5.30]{CEPSS}. The content is a sufficient condition for simplicity of the Steinberg algebra $\mathcal{A}_{\K}(\mathcal{G}_n)$.

\begin{lemma}\label{lem:the set Z zero} Let $\K$ be a field of characteristic zero and $\mathcal{G}_n$  a $\Z_2$-multispinal groupoid. If the support of every singular function on $\G_n$ is disjoint from $\{z_e \}$, then the ideal  $S_\K(\G_n)$ of singular functions is trivial.
\end{lemma}

\begin{proof} Having established Lemma~\ref{lem:f nonzero on ze is singular}, we note that
the proof of \cite[Lemma 5.30]{CEPSS} makes use  of the structure of the subset of the nucleus of the first Grigorchuk group determined by the image of $\F_{2^n}$, for $n=2$, and works in full generality for an arbitrary $\Z_2$-multispinal group.
	\end{proof}

We can now state the simplicity of the Steinberg algebra $\mathcal{A}_{\K}(\mathcal{G}_n)$.

	\begin{thm}\label{thm:simple Steinberg algebra multi} Suppose that $\K$ is a field  and $\mathcal{G}_n$ is a $\Z_2$-multispinal groupoid with Hausdorff unit spaces as above. If  $\chara(\K)=0$, then $A_\K(\G_n)$ is simple.
\end{thm}

\begin{proof} The result follows from  Lemmas~\ref{lem:f nonzero on ze is singular}  and \ref{lem:the set Z zero} in connection with  \cite[Theorem 3.14]{CEPSS}.
\end{proof}

\section{Simplicity of the  $C^*$-algebra of $\Z_2$-multispinal groups}\label{section:simplicity C*-algebra}

We now complete the proof that the $C^*$-algebras corresponding to the $\Z_2$-multispinal groups are simple, following the method outlined in \cite{CEPSS} for the Grigorchuk group.  First, we give an analog of Lemma 5.31 of \cite{CEPSS}:

\begin{lemma}
	\label{CEPSSLemma5.31}  Let $f\in \mathcal{A}_{\C}(\mathcal{G}_n),$ where $\mathcal{G}_n$ is the groupoid coming from a self-similar action corresponding to  any $\Z_2$- multispinal group $\mathfrak{G}_n,$ and suppose that there exists $m\in\N$ such that $f$ takes the form
	$$f\;=\;\sum_{g\in {\mathcal N}_{0,n}}c_g 1_{U_m(z_g)},\;\{c_g:\;g\in {\mathcal N}_{0,n}\}\subset \C,$$
	where recall that $\mathcal{N}=\mathcal{N}_{0,n}\sqcup\{a\}$ denotes the nucleus of $\mathfrak{G}_n.$  If $f(z_e)\not=0,$ then $|f(z)|>\frac{|f(z_e)|}{2^n}$ on a set of nonempty interior, i.e. $f\notin \mathcal{S}_{\C}(\mathcal{G}_n).$

\end{lemma}
\begin{proof}
Generalizing Lemma 5.31 of \cite{CEPSS}, the function $f$ is possibly nonzero on the following $2\cdot (2^{n})$ sets of ${\mathcal G}_n,$ which as in the disjointification process of  Lemma \ref{lem:primitive polynomial over Z two}, we label by $M_{\JJ_K},$ for
$$M_{\JJ_K}\;=\;\bigcap_{g\in K} U_m(z_g)\;=\;\bigcap_{g\in K} U_m(z_g)\setminus\bigcup_{h\in K^c}U_m(z_h),$$
where $K\in \{\Hf_{j,n}, \Hf_{j,n}^c\mid j=0,\dots,2^n-2\}$.
Enumerating the complex values that $f$ takes on these sets by $\kappa_1,\kappa_2,\;\cdots,\;\kappa_{2\cdot (2^{n-1})},$ we have, for $x_i\in M_{\JJ_{K_i}},$
  a system of $2\cdot (2^n-1)$ equations of the form
$$\sum_{g\in\mathcal{N}_{0,n}}c_g\;\cdot\;1_{U_m(z_g)}(x_i)\;=\;\kappa_i,\;1\;\leq\;i\leq \;\;2(2^{n-1}).$$
The $2(2^{n-1})\times 2(2^n-1)$ matrix with entries
$(1_{U_m(z_{g_j})}(x_i))$ is exactly  $W_n^t,$ the matrix transpose of that given in  Lemma \ref{primitivepolynomial_lemma_over_F2}.  So in matrix notation, we can write:
$$W_n^t (c_{e}=c_{g_0}, c_{g_2},\;\cdots,c_{g_{2^n-1}} )^t\;=\;(\kappa_1,\;\kappa_2,\cdots,\kappa_{2(2^{n-1})})^t,$$
where $\{g_0=e,g_1,\;\dots,g_{2^n-1}\}=\mathcal{N}_{0,n}\;=\;\iota_n(\F_{2^n}).$

But we have shown in Lemma \ref{lem:full rank general} that $W_n$ has a right-inverse $T_n,$ a $2(2^n-1)\times 2^n$ matrix $T_n$ with $W_nT_n=\text{I}_{2^n\times 2^n}.$
Therefore $T_n^tW_n^t=\text{I}_{2^n\times 2^n}.$  It follows that
$$(c_{e}, c_{g_1},\;\cdots,c_{g_{2^n-2}}, c_{g_{2^n-1}} )^t\;=\;T_n^tW_n^t(c_{e}, c_{g_1},\;\cdots,c_{g_{2^n-1}})^t\;=\;T_n^t(\kappa_1,\;\kappa_2,\cdots,\kappa_{2(2^{n-1})})^t.$$
Now suppose that for all  $i,\;1\;\leq\;i\;\leq 2^n,$ we have $\kappa_i\;\leq\;\frac{|f(z_e)|}{2^n}.$  Then, we have
$$f(z_e)\;=\;\sum_{j=1}^{2^n}(T_n^t)_{1,j}\kappa_i,$$
so that by the triangle inequality,
$$|f(z_e)|\;\leq\; \sum_{j=1}^{2^n} (T_n^t)_{1,j}|\kappa_i|\;<\;\sum_{j=1}^{2^n} \frac{1}{2^n}|\frac{|f(z_e)|}{2^n}<|f(z_e)|,$$
giving the desired contradiction.

\end{proof}

The next lemma is an analog of Lemma 5.32 of \cite{CEPSS}. The notation $\mathcal{B}(\mathcal{G}_n)$ stands for the normed space of  bounded complex valued functions on $\mathcal{G}_n$, endowed with the uniform norm, cf. \cite[Section 4.1]{CEPSS}.

\begin{lemma}
\label{Lemma5.32CEPSS}
Suppose that $f\in \mathcal{B}(\mathcal{G}_n),\; f(z_e)\not= 0,$  that $f_j\in \mathcal{A}_{\C}(\mathcal{G}_n)$ for all
$j\in\N,$ and that $f_j\;\to\; f$ uniformly. Then $\text{supp}(f)$ has nonempty interior.
\end{lemma}

\begin{proof}
	Suppose $f(z_e)=R\not=0.$  Find $N$ such that for all $j\geq N,\;\|f-f_j\|_{\infty}<\frac{R}{2^n}.$
	So $|f(z_e)-f_N(z_e)|=|R-f_N(z_e)|<\frac{|R|}{2^n},$ and by the reverse triangle inequality we obtain $\frac{(2^n-1)|R|}{2^n}<f_N(z_e).$  Note $|f_N(z)-0|>\frac{(2^n-1)|R|}{2^n}.$
	
	Now define $V\;=\;f_N^{-1}(f_N(z_e)).$  If the set $V$ has non-empty interior, we are done, since for all $z\in V,$ we have $f$ non-zero on $V$ as in the proof of \cite{CEPSS}, Lemma 5.32.  If the interior of $V$ is empty, applying Lemma \ref{lem:f restricts to combination of characteristic fcts of Um} shows that there exists $m\in\N$ and $c_{g_1}=c_e,\;c_{g_2},\;\cdots\;c_{g_{2^n-2}}$ with
	$$(f_N)_{|U_m}(z)\;=\;\sum_{i=1}^{2^n-2}c_{g_i}1_{U_m(z_{g_i}}(z),\;\text{for}\; U_m\;=\;\cup_{i=1}^{2^n-1}U_m(z_{g_i}).$$
	The method of Lemma \ref{CEPSSLemma5.31} then shows that there exists a set $Z$ with non-empty interior such that
	$$|f_N(z)|>\frac{|f_N(z_e)|}{2^n}\;>\;\frac{(2^n-1)|f(z_e)|}{2^{2n}},\;\forall z\in Z,$$
	thus completing the proof of the Lemma.
		\end{proof}
	
	Here is the analog of Lemma 5.33 of \cite{CEPSS}.
\begin{lemma}
	\label{Lemma5.33CEPSS} Suppose that $f\in \mathcal{B}(\mathcal{G}_n),$ with $f$ not equal to the zero function, and  that $f_j\in \mathcal{A}_{\C}(\mathcal{G}_n)$ for all
	$j\in\N,$ with $f_j\;\to\; f$ uniformly.  Then $\text{supp}(f)$ has nonempty interior.
	\end{lemma}
	
	\begin{proof}
	This proof is exactly the same as in Lemma 5.33 of \cite{CEPSS}, and we leave the details to the reader.
	\end{proof}
	
	We are finally in a position to prove the main theorem.
\begin{thm}
	Let $C^*(\mathcal{G}_n)$ be the full groupoid $C^*$-algebra coming from the groupoid $\mathcal{G}_n,$  arising from a self-similar action  of a $\Z_2$- multispinal group $\mathfrak{G}_n.$  Then $C^*(\mathcal{G}_n)$ is simple.
\end{thm}

\begin{proof} We know $\mathcal{G}_n$ is amenable, see \cite{Su}, so that $C^*(\mathcal{G}_n)\;\cong\;C^*_{\text{red}}(\mathcal{G}_n).$ We know that
$\mathcal{G}_n$ is minimal and effective from \cite{EP2}, Section 17. Lemma \ref{Lemma5.33CEPSS} has established that every non-zero element of $C^*_{\text{red}}(\mathcal{G}_n)$ has non-empty interior.  Therefore by Theorem 4.10 part 3 of \cite{CEPSS}, $C^*(\mathcal{G}_n)$ is simple.
\end{proof}

To conclude the paper we add a remark on the $L^p$ case.

\begin{remark} Our methods also imply simplicity  of the associated groupoid $L^p$-operator algebras $F^p(\G_n)$ (with the notation of  \cite{BKMc}). Indeed our proof also generalizes, as it is in the case of the Grigorchuk group, to the  $L^p$ setting, along the same lines; see  \cite[Example 7.44]{BKMc}).
\end{remark}

\end{document}